\documentclass{amsart}
\usepackage{amsmath,amsthm,amssymb}
\usepackage{lipsum} 
\usepackage[utf8]{inputenc}
\usepackage{esint}
\usepackage[all]{xy}
\usepackage{amsfonts} 
\usepackage{xcolor}
\usepackage{tikz}
\usepackage{color}
\usepackage[colorlinks=true]{hyperref}
\usepackage{mathrsfs}
\hypersetup{
    linkcolor=cyan,          % color of internal links
    citecolor=orange,        % color of links to bibliography
    filecolor=cyan,      % color of file links
    urlcolor=cyan
}
\usepackage{bbold}
\usepackage{enumerate}
\usepackage{calc}
\usepackage{bbold}
\usepackage[colorinlistoftodos,prependcaption,textsize=tiny]{todonotes}
\usepackage{todonotes}
\usepackage[shortcuts]{extdash}
\usepackage[style=alphabetic,maxnames=200,maxalphanames=6,giveninits]{biblatex}
\makeatletter
\numberwithin{equation}{section}
\theoremstyle{plain}
        \newtheorem{theorem}{Theorem}[section]

        \newtheorem{assumption}[theorem]{Assumption}
        
        \newtheorem{lemma}[theorem]{Lemma}
         
        \newtheorem{definition}[theorem]{Definition} 
        \newtheorem{remark}[theorem]{Remark} 
         \newtheorem{claim}[theorem]{Claim} 

\usepackage[shortcuts]{extdash}

\numberwithin{equation}{section}

\def\R{\mathbb{R}}
\def\T{\mathbb{T}}

\def\div{\mathrm{div}}

\def\cD{\mathcal{D}}

\def\cH{\mathcal{H}}

\def\logl{ L\log L}

\def\hf{\mathbb f}

\def\Z{{\mathbb Z}}
\def\G{\R^{4d}}

\def\supp{\operatorname{supp}}

\newcommand{\defeq }{\mathop{=}\limits^{\textrm{def}}}
\def\d{\partial}
\def\Do{\R^{2d}}

\def\Dot{\T^d\times\R^d}
\newcommand{\dd}{{\,\rm d}}

\usepackage{todonotes}

\title{On a fuzzy Landau Equation:\ II. Solvability results}



\author{Manh Hong Duong
  \and  Zihui He
}

\newcommand{\Addresses}{{% additional braces for segregating \footnotesize
  \bigskip
  \footnotesize

  H.~Duong, \textsc{School of Mathematics, University of Birmingham, UK}\par\nopagebreak
  \textit{E-mail}: \texttt{h.duong@bham.ac.uk}

  \medskip

  Z.~He, \textsc{Fakult\"at f\"ur Mathematik, Universit\"at Bielefeld, Postfach 100131, 33501 Bielefeld, Germany}\par\nopagebreak
  \textit{E-mail}: \texttt{zihui.he@uni-bielefeld.de}

}}
\begin{document}

\maketitle

\begin{abstract}
This article is the second in a series of works on the fuzzy Landau equation, which describes particles interacting through delocalised Coulomb-type collisions. The spatial interaction is modulated by a kernel. We establish the existence of global-in-time $\mathcal{H}$-solutions for exponentially decaying spatial kernels on the whole space and for kernels bounded above and below by positive constants, including periodised kernels on the torus. For the latter class, we further study the propagation of velocity moments and $L^p$-bounds.
\end{abstract}

\tableofcontents

\section{Introduction}\label{sec:intro}
In this paper, we consider the spatially inhomogeneous (kinetic) \textit{fuzzy Landau equation}
\begin{equation}
    \label{Landau-fuz}
    \begin{aligned}
     &\d_t f+v\cdot\nabla_x f=Q(f,f)\\
     & Q(f,f)=\nabla_v \cdot \int_{\Do}\kappa(x-x_*)A(|v-v_*|)\Pi_{(v-v_*)^\perp}\big(f_*\nabla_v f-f\nabla_{v_*} f_*\big)\dd x_*\dd v_* .
    \end{aligned}
\end{equation}
The fuzzy Landau equation characterises the time evolution of the unknown $f=f_t(x,v):[0,T]\times\Do\to\R_+$, $d\ge 2$, which describes the distribution of particles in a plasma at time $t$ at position $x$ and with velocity $v$, accounting for both free transport and particle collisions. The spatial domain can be $\T^d$ or $\R^d$. We write $f_*=f(x_*,v_*)$ denotes the density of particle at different position and velocity $(x_*,v_*)$.

In the above, $\Pi_{z^\perp}$ denotes the orthogonal projection onto $z^\perp$,
\begin{align*}
    \Pi_{z^\perp}=\operatorname{Id}-\frac{z\otimes z}{|z|^2},
    \quad z\ne0,
\end{align*}
while $A(|v-v_*|)\ge0$ denotes a kinetic kernel and
$\kappa(x-x_*)\ge0$ denotes a spatial kernel, which will be explicitly specified in Section \ref{sec:pre}.

In the first paper in this series \cite{DH25} and in \cite{DH26}, we showed that the fuzzy Landau equation \eqref{Landau-fuz} can be rigorously cast into the so-called \textit{GENERIC framework} via a variational characterisation. Both frameworks concern curves satisfying
\begin{align*}
    f\in L^\infty\big([0,T];L^1_{2,2+\gamma_+}
    \cap\logl(\Do;\R_+)\big).
\end{align*}
The norm $\|f\|_{L^1_{a,b}}:=\|(\langle x\rangle^a+\langle v\rangle^b)f\|_{L^1_{x,v}}$. 
\begin{itemize}
    \item \underline{Framework in \cite{DH25}:}
    The formulation in \cite{DH25} is based on the Landau gradient
    \begin{equation}
    \label{grad-1}
        \widetilde\nabla f
        =\sqrt{A(|v-v_*|)}\Pi_{(v-v_*)^\perp}
        (\nabla_vf-\nabla_{v_*}f_*),
    \end{equation}
    with kernels satisfying
    \begin{gather}
        A(|v-v_*|)=|v-v_*|^{2+\gamma},
        \quad \gamma\in(-\gamma_d,1],\notag\\
        \kappa(x)\sim1
        \quad\text{on $\R^d$ or $\T^d$},\label{sk-1}
    \end{gather}
    where $\gamma_d=\min(d,4)$. 
    In addition, we require, for some
    $p\in\big(\frac{d}{d-|\gamma|},+\infty\big)$,
    \begin{equation}
        \label{require}
    \begin{gathered}
        \langle v\rangle^{2+|\gamma|}f
        \in L^\infty([0,T];L^p_vL^1_x),
        \quad \gamma\in(-\gamma_d,0),\\
        f\in L^\infty([0,T];L^p_vL^1_x),
        \quad \gamma\in(0,1].
    \end{gathered}
    \end{equation}
    \item \underline{Framework in \cite{DH26}:}
    In \cite{DH26}, in stead of \eqref{grad-1}, we use a cross-product Landau gradient in $\mathbb{R}^3$
    \begin{align}
    \label{grad-2}
\widetilde{\boldsymbol{\nabla}} f
        =(v-v_*)\times
        \big(\nabla_vf(x,v)-\nabla_{v_*}f(x_*,v_*)\big),
    \end{align}
    with the cross product replaced by its skew-symmetric matrix
    analogue in higher dimensions. This formulation allows kernels satisfying
    \begin{gather}
        A(|v-v_*|)\sim|v-v_*|^2\langle v-v_*\rangle^\gamma,
        \quad \gamma\in(-\infty,1],\notag\\
        \kappa(x)\sim\exp(-C\langle x\rangle),
        \quad x\in\R^d.\label{sk-2}
    \end{gather}
\end{itemize}
The formulation in \cite{DH25} relies on the positive lower and upper
bounds on $\kappa$, while that in \cite{DH26} relies on the regularised
structure of $A$.
Our goals in this paper are 
\begin{enumerate}[(1)]
    \item To establish existence results for both types of kinetic kernels,
    combined with either bounded or spatially decaying kernels.

    \item To establish the propagation of moments and $L^p$-estimates
    for the kernels in \eqref{sk-1}.
\end{enumerate}
The precise assumptions on $A$ and $\kappa$ are given in
Assumption \ref{ASS:kernel} below.

The present paper is part of our series of papers that study the fuzzy Landau equation \eqref{Landau-fuz}: \cite{DH25, DH26} develop
the variational framework, the present paper studies existence
and propagation of regularity, and \cite{DGH25} establishes
the grazing collision limit from the fuzzy Boltzmann equation
to the fuzzy Landau equation.

\subsection{Main results}

Inspired by the definition of $\cH$-solutions for homogeneous equations in \cite{Vil98b}, we consider $\cH$-solutions to the fuzzy Landau equation \eqref{Landau-fuz}, defined in Definition \ref{def:H} below. This concept is based on the Boltzmann entropy
\begin{align*}
    \cH(f)=\int_{\Do}f\log f\dd x\dd v,
\end{align*}
and the entropy dissipation
\begin{equation*}
    \cD(f)=\frac12\int_{\G}\kappa A ff_*
    \big|\Pi_{(v-v_*)^\perp}
    \big(\nabla_v\log f-\nabla_{v_*}\log f_*\big)\big|^2
    \ge0.
\end{equation*}
For convenience, we use $\Do$ to denote either
$\R^d\times\R^d$ or $\T^d\times\R^d$, with spatial integrals
over $\R^d$ or $\T^d$, as appropriate.
For both bounded and spatially decaying kernels, as described in
\eqref{sk-1} and \eqref{sk-2}, we establish the existence of
global-in-time $\cH$-solutions.
\begin{theorem}
\label{mainthm1}
Let $A$ and $\kappa$ satisfy Assumption \ref{ASS:kernel}.
Let $f_0\in L^1_{2,2+\gamma_+}(\Do)\cap\logl(\Do;\R_+)$.
Then there exists an $\cH$-solution
$f\in L^\infty([0,T];L^1_{2,2+\gamma_+}(\Do))$
to the fuzzy Landau equation \eqref{Landau-fuz} with initial value
$f_0$. Moreover, the following entropy inequality holds
\begin{equation*}
    \begin{aligned}
        \cH(f_t)-\cH(f_0)+\int_0^t\cD(f_s)\dd s\le0\quad \text{for all}\quad t\in[0,T].
    \end{aligned}
\end{equation*}
\end{theorem}

We prove Theorem \ref{mainthm1} in Section \ref{sec:existence}.
In Section \ref{sec:reg}, we study the propagation of moments
and $L^p$-norms of $\cH$-solutions for spatial kernels
$\kappa\gtrsim1$. More precisely,
\begin{itemize}
    \item Lemmas \ref{lem:M-1} and \ref{lem:M-2}
    establish the propagation of velocity moments for hard,
    moderately soft and very soft potentials, under the
    assumptions stated there.
    \item Lemma \ref{lem:Lp-2} establishes the propagation of
    $L^p_{x,v}$-norms for moderately soft potentials.
    \item Lemma \ref{lem:Lp-3} implies estimates in the mixed norms
    $L^p_vL^1_x$ for hard potentials when $\kappa\equiv1$.
\end{itemize}
We extend techniques developed for the spatially homogeneous
Landau equation in \cite{Vil98b,DV00a,Des15,TV00,ALL15}
to account for the spatial variable and the transport term,
working in dimensions $d\ge2$.

We compare these results with the requirements in \eqref{require}
in Remark \ref{rmk:comparison} below.

\subsection{Related works}

In \eqref{sk-1} on $\T^d$ and \eqref{sk-2} on $\R^d$,
the spatial kernels can be taken in the following form, with
a fixed $\sigma\in(0,1)$:
\begin{equation*}
\begin{gathered}
    \kappa^\sigma(x)
    =C\sigma^{-\frac d2}
    \sum_{n\in\Z^d}
    \exp\Big(-\Big\langle\frac{x+n}{\sqrt{\sigma}}\Big\rangle\Big),
    \quad x\in\T^d,\\
    \kappa^\sigma(x)
    =C\sigma^{-\frac d2}
    \exp\Big(-\Big\langle\frac{x}{\sqrt{\sigma}}\Big\rangle\Big),
    \quad x\in\R^d,
\end{gathered}
\end{equation*}
where $C>0$ is independent of $\sigma$ and is chosen to
normalise the spatial mass to one.
The fuzzy Landau equation can be viewed as an intermediate model
between the spatially homogeneous Landau equation
\begin{equation}
\label{Landau-homo}
\begin{aligned}
    &\partial_tf=Q_{\sf homo}(f,f),
    \quad f_t:\R^d\to\R_+,\\
    &Q_{\sf homo}(f,f)
    =\nabla_v\cdot\int_{\R^d}A\Pi_{(v-v_*)^\perp}
    \big(f(v_*)\nabla_vf(v)-f(v)\nabla_{v_*}f(v_*)\big)
    \dd v_*
\end{aligned}
\end{equation}
and the classical inhomogeneous Landau equation
\begin{equation}
\label{col:Landau-cla}
\begin{aligned}
    &\d_tf+v\cdot\nabla_xf=Q_{\sf cl}(f,f), \quad f_t:\R^{2d}\to\R_+,\\
    &Q_{\sf cl}(f,f)
    =\nabla_v\cdot\int_{\R^d}A\Pi_{(v-v_*)^\perp}
    \big(f(x,v_*)\nabla_vf(x,v)
    -f(x,v)\nabla_{v_*}f(x,v_*)\big)\dd v_*.
\end{aligned}
\end{equation}
Indeed, as $\sigma\to0$, the spatial kernel $\kappa^\sigma$
converges weakly to a Dirac measure, and the fuzzy Landau
equation \eqref{Landau-fuz} formally converges to the classical
inhomogeneous Landau equation.
In particular, the periodised
kernel on the unit torus tends to $1$ as $\sigma\to+\infty$.
When $\kappa\equiv1$, integrating the fuzzy Landau equation
over the spatial variable implies the homogeneous Landau equation
for the velocity marginal, at least on a formal level.

The classical Landau equations are fundamental partial differential
equations in kinetic theory and have been studied extensively
in both physics and mathematics
\cite{cercignani1969mathematical,Vil98,Vil98b}.
The collision operator in \eqref{col:Landau-cla} is local in
the spatial variable: two particles with velocities $v$ and
$v_*\in\R^d$ collide at the same position $x$.
In the fuzzy Landau equation \eqref{Landau-fuz}, collisions
are delocalised: particles at different states $(x,v)$ and
$(x_*,v_*)$ can interact, with their positions coupled through
the spatial kernel $\kappa$.

To the best of our knowledge, the fuzzy Landau equation
\eqref{Landau-fuz} was first introduced and studied in \cite{DH25}, the first
paper in this series. It has been investigated further by
\textcite{gualdani2025fuzzylandauequationglobal}.
Other models featuring delocalised collision operators have
been considered in the literature, including the Boltzmann
equation \cite{EH25,EH25b}, the Maxwell--Boltzmann equation
\cite{morgenstern1955analytical}, the Enskog equation
\cite{arkeryd1986enskog,arkeryd1990global} and the Povzner
equation \cite{povzner1962boltzmann}.
We now compare our results with the mathematical literature
most directly relevant to this paper.

\textit{On classical Landau equations.}
The equations most closely related to the fuzzy Landau equation
\eqref{Landau-fuz} are the spatially homogeneous and inhomogeneous
Landau equations \eqref{Landau-homo} and \eqref{col:Landau-cla}.
For the homogeneous model, the existence of weak and
$\cH$-solutions, uniqueness and propagation of regularity
have been studied, for example, in
\cite{Vil98,Vil98b,DV00a,DV00b,Wu14,ALL15,FG09}.
Recent work rules out finite-time blow-up for suitable interaction
kernels, including the Coulomb potential, by exploiting the
monotonicity of Fisher information; see \cite{NL23,NL25,Vil25}.
For the inhomogeneous Landau equation, global-in-time renormalised
solutions with a defect measure were constructed in \cite{Vil96}.
Global-in-time smooth solutions for general large initial data
in the potential range considered here remain an open problem.

\textit{On strong solutions to the fuzzy Landau equation.}
In their preprint,
\textcite{gualdani2025fuzzylandauequationglobal}
establish global-in-time existence and uniqueness of smooth
solutions in dimension three for $\gamma\in(-\sqrt8,0]$
and spatial kernels
$\kappa(x)=\langle x\rangle^{-\lambda}$, $\lambda\ge0$,
under suitable regularity and moment assumptions on the initial
value. This includes $\kappa\equiv1$.
They also obtain Fisher-information estimates for a broader
class of spatial kernels, including exponentially decaying
kernels. Our results concern $\cH$-solutions under weaker
initial regularity assumptions and cover dimensions $d\ge2$,
with the kernel assumptions specified below.

\textit{On the grazing collision limit.}
The Landau equations can be obtained from the corresponding
Boltzmann equations in the grazing collision limit, where
collisions are dominated by small deflection angles; see,
for example, \cite{Vil98b,carrillo2022boltzmann}.
In the third paper in this series \cite{DGH25}, we derive
the fuzzy Landau equation \eqref{Landau-fuz} from the fuzzy
Boltzmann equation \cite{EH25} through this limit.

\subsection*{Organisation of the paper}

The rest of the paper is organised as follows.
In Section \ref{sec:pre}, we introduce the notation, the precise
assumptions on $A$ and $\kappa$, and the functional inequalities
used throughout the paper.
In Section \ref{sec:existence}, we prove the existence of
$\cH$-solutions stated in Theorem \ref{mainthm1}.
In Section \ref{sec:reg}, we study the propagation of moments
and $L^p$-norms of these solutions.

\subsection*{Acknowledgements}
M. H. D is funded by an EPSRC Standard Grant EP/Y008561/1. Z.~H. is funded by the Deutsche Forschungsgemeinschaft (DFG, German Research Foundation) – Project-ID 317210226 – SFB 1283.

\section{Preliminary}\label{sec:pre}

In Section \ref{subsec:kernel}, we present the precise assumptions on the kernels $A$ and $\kappa$. In Section \ref{subsec:notation}, we introduce the notation and functional inequalities used throughout the paper. In Section \ref{sec:dissipation}, we recall the hydrodynamic bounds and weighted Fisher-information estimate, and derive the weighted $L^p$-estimates used below.

Throughout this article, the time $T\in(0,+\infty)$ is arbitrary and fixed. The constant $C>0$ may vary from line to line. We use the notation $B^R_x:=\{x\in\R^d\mid |x|\le R\}$,
and define $B_v^R$ and $B_{x,v}^R$ analogously.

\subsection{Collision and spatial kernels}\label{subsec:kernel}

The kinetic kernel $A$ and the spatial kernel $\kappa$
satisfy the following assumptions.

\begin{assumption}[Kernel assumptions]
\label{ASS:kernel}
\begin{itemize}

\item \underline{\it Kinetic kernels.}
We consider the interaction kernel $A:(0,\infty)\to\R_+$ given by
\begin{equation}
\label{A-1}
    A(|v-v_*|)=|v-v_*|^{2+\gamma},
    \quad \gamma\in(-\gamma_d,1].
\end{equation}
We also consider kernels satisfying
\begin{equation}
\label{A-3}
    0<C^{-1}\langle v-v_*\rangle^\gamma
    \le\frac{A(|v-v_*|)}{|v-v_*|^2}
    \le C\langle v-v_*\rangle^\gamma,
    \quad \gamma\in(-\infty,1].
\end{equation}
In addition, for $\gamma<0$, $A$ is sufficiently smooth, and $A'$ satisfies \eqref{A3-2} below.

\smallskip

\item \underline{\it Spatial kernels with positive lower and upper bounds.}
We consider spatial kernels satisfying
\begin{equation}
\label{k-1}
    0<C^{-1}\le\kappa(x)\le C.
\end{equation}
For example, one can take $\kappa\equiv1$.
On $\T^d=\R^d/\Z^d$, we also consider, for any fixed
$\sigma\in(0,1)$, the periodised kernel
\begin{equation}
\label{k-3}
    \kappa^\sigma(x)
    =C\sigma^{-\frac d2}
    \sum_{n\in\Z^d}
    \exp\Big(-\Big\langle\frac{x+n}{\sqrt{\sigma}}\Big\rangle\Big),
\end{equation}
where $C>0$ is independent of $\sigma$ and is chosen such that
$\|\kappa^\sigma\|_{L^1(\T^d)}=1$.
For each fixed $\sigma$, this kernel also satisfies \eqref{k-1},
with bounds that may depend on $\sigma$.

\smallskip

\item \underline{\it Spatial kernels vanishing at infinity.}
On $\R^d$, we also consider, for any fixed $\sigma\in(0,1)$,
the spatial kernel
\begin{equation}
\label{k-2}
    \kappa^\sigma(x)
    =C\sigma^{-\frac d2}
    \exp\Big(-\Big\langle\frac{x}{\sqrt{\sigma}}\Big\rangle\Big),
\end{equation}
where $C>0$ is independent of $\sigma$ and is chosen such that
$\|\kappa^\sigma\|_{L^1(\R^d)}=1$.

We always assume that $\kappa$ is even $\kappa(x)=\kappa(-x)$.

\end{itemize}
\end{assumption}
\begin{remark}
    For kernels satisfying \eqref{A-3}, there is no singularity in $A$
near $v=v_*$.
The only genuinely very soft potential case considered here arises
in \eqref{A-1}, where
\begin{align*}
    A=|v-v_*|^{2+\gamma},
    \quad \gamma\in(-\gamma_d,-2).
\end{align*}

As $\sigma\to0$, $\kappa^\sigma$ in \eqref{k-3} and \eqref{k-2}  converge weakly to the
Dirac measure at the origin.
\end{remark}
\subsection{Notations}
\label{subsec:notation}

We use $\Do$ for either $\Do$ or $\Dot$, with spatial integrals over
$\R^d$ or $\T^d$, as appropriate. We write
$\langle z\rangle:=\sqrt{1+|z|^2}$,
$\gamma_+=\max(0,\gamma)$ and
$\dd\eta=\dd x_*\dd v_*\dd x\dd v$.

For $a,b\in\R$ and $p,q\in[1,\infty]$, we use
\begin{align*}
    \|f\|_{L^1_{a,b}(\Do)}
    &\defeq\int_{\Do}
    \big(\langle x\rangle^a+\langle v\rangle^b\big)|f|\dd x\dd v,\\
    \|f\|_{L^p_vL^q_x}
    &\defeq\big\|\|f\|_{L^q_x}\big\|_{L^p_v},
    \quad
    \|f\|_{L^q_xL^p_v}
    \defeq\big\|\|f\|_{L^p_v}\big\|_{L^q_x}.
\end{align*}
For $s\ge0$, we define
\begin{equation*}
\label{moments}
    M_s(f)=\int_{\Do}\langle v\rangle^sf\dd x\dd v.
\end{equation*}
For $f\in L^1(\Do;\R_+)$, we define
\begin{align*}
    \cH(f)&=\int_{\Do}f\log f\dd x\dd v,\\
    \logl(\Do)
    &=\big\{f\in L^1(\Do;\R_+)\mid f\log f\in L^1(\Do)\big\},
\end{align*}
where $\cH(f)=+\infty$ if $f(\log f)_+$ is not integrable.

We note that $f\in L^1_{2,2}(\Do)$ and $\cH(f)<+\infty$ ensures 
$f\in \logl(\Do)$.

\subsection{Entropy dissipation and \texorpdfstring{$L^p$}{TEXT}-bounds}
\label{sec:dissipation}

Following \cite{DH25}, we define the fuzzy Landau gradient and
entropy dissipation by
\begin{equation*}
    \widetilde\nabla f
    =\sqrt{A}\Pi_{(v-v_*)^\perp}
    (\nabla_vf-\nabla_{v_*}f_*),
\end{equation*}
and
\begin{equation*}
\cD(f)=\frac12\int_{\G}\kappa ff_*
    |\widetilde\nabla\log f|^2\dd\eta.
\end{equation*}
 We use $\widetilde\nabla$ given by \eqref{grad-1} as notation for the Landau gradient operator. The analysis here is independent of the choice of Landau gradient \eqref{grad-1} or \eqref{grad-2}.
 
For solutions to \eqref{Landau-fuz}, we use the entropy inequality
\begin{equation*}
\label{H}
    \cH(f_T)-\cH(f_0)+\int_0^T\cD(f_t)\dd t\le0.
\end{equation*}
If $\cH(f_0)<+\infty$ and $f_T\in L^1_{2,2}(\Do)$, it also implies that
\begin{equation*}
\label{D}
    \int_0^T\cD(f_t)\dd t<+\infty.
\end{equation*}
Below, we normalize the mass to one and use the bounds established
in Section \ref{sec:existence}:
\begin{align*}
    \|f_t\|_{L^1(\Do)}=1,
    \quad
    \|\langle v\rangle^2f_t\|_{L^1(\Do)}\le E_0
    \quad\forall t\in[0,T].
\end{align*}

For the hydrodynamic estimates recalled from \cite{DH25}, we write
$$
    \hf(x,v)=\int_{x_*}f(x_*,v)\kappa(x-x_*)\dd x_*,
$$
and
\begin{align*}
    m(f)=\int_{v}f(x,v),\quad 
    E(f)=\int_{v}\langle v\rangle^2f(x,v),\quad 
    H(f)=\int_{v}(\log f)_+f(x,v).
\end{align*}

The following bounds are recalled from \cite{DH25}.
\begin{lemma}[\cite{DH25}, Section 2.4]\label{lem:hydro}
Let $\kappa$ satisfy Assumption~\ref{ASS:kernel}.
Let $f\in L^1_{2,2}\cap \logl(\Do;\R_+)$ satisfy
\begin{align*}
    \|\langle v\rangle^2f\|_{L^1_{x,v}}\le E_0,
    \quad
    \|(\log f)_+f\|_{L^1_{x,v}}\le H_0
\end{align*}
for some $E_0,\,H_0>0$.
\begin{itemize}
    \item \textbf{(Hydrodynamic bounds).} We have the following bounds.
\begin{equation*}
    E(\hf)\le\|\kappa\|_{L^\infty}E_0.
\end{equation*}
\begin{enumerate}[(1)]

\item \underline{\it Spatial kernels bounded away from zero.}
If $\kappa\sim 1$ is given as in \eqref{k-1} or \eqref{k-3}, then 
\begin{equation*}
    m(\hf)\gtrsim_\kappa\|f\|_{L^1_{x,v}}.
\end{equation*}
For $x\in\R^d$, we have
\begin{equation*}
    H(\hf)\lesssim_\kappa H_0+\|f\|_{L^1_{2,2}},
\end{equation*}
while for $x\in\T^d$, we have
\begin{equation*}
    H(\hf)\lesssim_\kappa1+H_0.
\end{equation*}

\item \underline{\it Spatial kernels vanishing at infinity.}
If $\kappa$ takes the form of Gaussian on $\R^d$ as in \eqref{k-2}, then
\begin{equation*}
    H(\hf)\le\|\kappa\|_{L^\infty}H_0,
    \quad
    \|\hf(x,\cdot)\|_{L^p_v}
    \le\|\kappa\|_{L^\infty}\|f\|_{L^p_vL^1_x}
    \quad\forall p\in(1,\infty).
\end{equation*}
Furthermore,
\begin{equation*}
    m(\hf)\ge
    C(f,\kappa)
    \exp\Big(-\Big\langle\frac{x}{\sqrt{\sigma}}\Big\rangle\Big)
    \|f\|_{L^1_{x,v}},
\end{equation*}
where $C(f,\kappa)>0$ is independent of $x$.

\end{enumerate}

\item\textbf{(Weighted Fisher information).}
Let $A$ satisfy Assumption~\ref{ASS:kernel}. Let $\cD(f)\in L^1([0,T])$. We have
\begin{equation}
\label{ineq:Des}
    \int_0^T\int_{B^R_x\times\R^d_v}
    \langle v\rangle^\gamma f_t
    \Big|\frac{\nabla_vf_t}{f_t}\Big|^2
    \le C\Big(1+\int_0^T\cD(f_t)\dd t\Big),\quad \forall\, R>0,
\end{equation}
where the constant $C>0$ depends on $d,\gamma,\kappa,E_0,H_0,T$, and also on
$\|f\|_{L^\infty_tL^1_{2,0}}$ in the case of $x\in\R^d$.
\begin{enumerate}[(1)]

\item \underline{\it Spatial kernels bounded away from zero.} If $\kappa\sim 1$ is given as in \eqref{k-1} or \eqref{k-3}, then $C$ is independent of $R$.
In particular, \eqref{ineq:Des} holds on the full phase space $\Do$ or $\Dot$.

\item \underline{\it Spatial kernels vanishing at infinity.}
If $\kappa$ takes the form of Gaussian on $\R^d$ as in \eqref{k-2}, then $C=C(R)\to\infty$ as $R\to+\infty$. 

\end{enumerate}

\end{itemize}

\end{lemma}

The next lemma collects the estimates used to control the collision terms.
Its first part relates weighted $L^p$-norms to Fisher information.
Parts (2) and (3) interpolate these norms with velocity moments,
and parts (4) and (5) control the singular convolution integrals.

\begin{lemma}\label{lem:sobo}
Let $\gamma\in(-\gamma_d,0]$.
Let $k=\frac{d}{d-2}$ if $d\ge3$, and $k\in(1,\infty)$ if $d=2$.
The functions $f$ and $g$ below are nonnegative, sufficiently smooth,
and decay sufficiently fast at infinity.
\begin{enumerate}[(1)]

\item (Sobolev embedding, \cite{BCD11})
Let $\varphi\in H^1(\R^d)$. We have
\begin{equation*}
\label{sobo}
    \Big(\int_{\R^d}|\varphi|^{2k}\dd v\Big)^{\frac1{2k}}
    \lesssim_{d,k}
    \Big(\int_{\R^d}|\nabla_v\varphi|^2\dd v\Big)^{\frac12}
    +\Big(\int_{\R^d}|\varphi|^2\dd v\Big)^{\frac12}.
\end{equation*}
When $d\ge3$, the last term can be omitted.

\item (\cite[Corollary 2.13]{DH25})
Let $\alpha\in(0,d)$, $\delta\in\R$ and
$p>\frac{d}{d-\alpha}$. Then
\begin{equation*}
\begin{aligned}
    \sup_{v\in\R^d}
    \|(\langle v\rangle^\delta f)*_v|v|^{-\alpha}\|_{L^1_x}
    \lesssim{}&
    \|\langle v\rangle^\delta f\|_{L^1(\Do)}\\
    &+\min\big(
    \|\langle v\rangle^\delta f\|_{L^1_xL^p_v},
    \|\langle v\rangle^\delta f\|_{L^p_vL^1_x}
    \big).
\end{aligned}
\end{equation*}

\item Let $\kappa$ be given by \eqref{k-1} or \eqref{k-3}.
Then
\begin{equation}
\label{ineq:w:k}
    \|\langle v\rangle^\gamma f\|_{L^1_tL^1_xL^k_v}
    \le C\Big(1+\int_0^T\cD(f_t)\dd t\Big),
\end{equation}
where $C$ has the same parameter dependence as in
Lemma \ref{lem:hydro}.
For $\kappa$ given by \eqref{k-2}, the corresponding estimate
holds locally in $x$, with $C$ also depending on the spatial domain $B^R_x$, $R>0$.

\item Let $s>2$, $\beta\in(0,1)$,
$\frac1p=\frac{\beta}{k}+(1-\beta)$ and
$\alpha=|\gamma|\beta+(\beta-1)s$.
Then
\begin{align*}
    \|\langle v\rangle^{-\alpha}f\|_{L^1_xL^p_v}
    \le
    \|\langle v\rangle^\gamma f\|_{L^1_xL^k_v}^{\beta}
    M_s(f)^{1-\beta}.
\end{align*}

\item Let $0<\varepsilon<\min(1,k-1)$ and
$s=|\gamma|k\big(\frac{k-1}{\varepsilon}-1\big)$.
Then
\begin{align*}
    \|f\|_{L^1_xL^{k-\varepsilon}_v}
    \le
    \|\langle v\rangle^\gamma f\|_{L^1_xL^k_v}^{\frac{k(k-1-\varepsilon)}{(k-\varepsilon)(k-1)}}
    M_s(f)^{\frac{\varepsilon}{(k-\varepsilon)(k-1)}}.
\end{align*}

\item Let $\alpha\in(0,2)$ and $B^R_v=\{|v|\le R\}$.
When $d=2$, assume $k>\frac2{2-\alpha}$.
Then
\begin{equation}
\label{ineq-3}
    \int_0^T\int_{\G\cap\{|v|\le R,\ |v_*|\le R\}}
    ff_*|v-v_*|^{-\alpha}\dd\eta\dd t
    \lesssim_{\alpha,d,k,R}
    \|f\mathbb{1}_{B^R_v}\|_{L^1_TL^1_xL^k_v}
    \|f\|_{L^\infty_TL^1_{x,v}}.
\end{equation}

\item Let $\alpha\in(0,2)$ and
$\varepsilon\in(0,k-\frac{d}{d-\alpha})$.
Let $f\in L^1_TL^1_xL^{k-\varepsilon}_v$
and $g\in L^\infty_TL^1_{x,v}$.
Then
\begin{equation}
\label{ineq-4}
\begin{aligned}
&\int_0^T\int_{\G}fg_*|v-v_*|^{-\alpha}\dd\eta\dd t\\
&\lesssim_{\alpha,d,k,\varepsilon}
\big(
\|f\|_{L^1_TL^1_{x,v}}
+\|f\|_{L^1_TL^1_xL^{k-\varepsilon}_v}
\big)
\|g\|_{L^\infty_TL^1_{x,v}}.
\end{aligned}
\end{equation}

\end{enumerate}
\end{lemma}

\begin{proof}
\begin{enumerate}[(1)]

 \item This is the standard Sobolev embedding; see \cite{BCD11}.

 \item The proof can be found in \cite{DH25}.

\item We follow \cite[Lemma 3]{Des15}.
Applying (1) to $\langle v\rangle^{\frac{\gamma}{2}}\sqrt f$,
squaring the resulting inequality and integrating over $x$, we obtain
\begin{align*}
    \|\langle v\rangle^\gamma f\|_{L^1_xL^k_v}
    &\lesssim
    \int_{\Do}
    \Big|\nabla_v\big(\langle v\rangle^{\frac{\gamma}{2}}\sqrt f\big)\Big|^2
    \dd x\dd v
    +\int_{\Do}\langle v\rangle^\gamma f\dd x\dd v\\
    &\lesssim
\int_{\Do}\langle v\rangle^\gamma
|\nabla_v\sqrt f|^2\dd x\dd v
+\|f\|_{L^1_{x,v}}.
\end{align*}
Lemma \ref{lem:hydro}, together with its stated linear dependence
on the higher velocity moment, yields \eqref{ineq:w:k}.
For $\kappa$ given by \eqref{k-2}, the same argument applies on arbitrary bounded spatial domains $B^R_x$, $R>0$.

\item The relations between $\alpha,\beta,p$ and $s$ ensure
\begin{align*}
    \langle v\rangle^{-\alpha}f
    =(\langle v\rangle^\gamma f)^\beta
    (\langle v\rangle^sf)^{1-\beta}.
\end{align*}
By Hölder's inequality in $v$, followed by Hölder's inequality in $x$,
we obtain
\begin{align*}
    \|\langle v\rangle^{-\alpha}f\|_{L^1_xL^p_v}
    &\le
    \int_x
    \|\langle v\rangle^\gamma f\|_{L^k_v}^{\beta}
    \|\langle v\rangle^sf\|_{L^1_v}^{1-\beta}\dd x\\
    &\le
    \|\langle v\rangle^\gamma f\|_{L^1_xL^k_v}^{\beta}
    M_s(f)^{1-\beta}.
\end{align*}

\item We take
\begin{align*}
    q_1=\frac{k-1}{k-1-\varepsilon},
    \quad
    q_2=\frac{k-1}{\varepsilon}.
\end{align*}
Then $\frac1{q_1}+\frac1{q_2}=1$, and the choice of $s$ implies
\begin{align*}
    f^{k-\varepsilon}
    =(\langle v\rangle^\gamma f)^{\frac{k}{q_1}}
    (\langle v\rangle^sf)^{\frac1{q_2}}.
\end{align*}
Hölder's inequality yields
\begin{align*}
    \|f\|_{L^{k-\varepsilon}_v}
    \le
    \|\langle v\rangle^\gamma f\|_{L^k_v}^{\frac{k}{(k-\varepsilon)q_1}}
    \|\langle v\rangle^sf\|_{L^1_v}^{\frac1{(k-\varepsilon)q_2}}.
\end{align*}
Since
$\frac{k}{(k-\varepsilon)q_1}
+\frac1{(k-\varepsilon)q_2}=1$,
another application of Hölder's inequality in $x$ implies the claim.

\item Let $\frac1k+\frac1{k'}=1$.
By Hölder's and Young's convolution inequalities, we have
\begin{align*}
&\int_{B^R_v\times B^R_{v_*}}
ff_*|v-v_*|^{-\alpha}\dd v_*\dd v\\
&\le
\|f(x,\cdot)\mathbb{1}_{B^R_v}\|_{L^k_v}
\|(f(x_*,\cdot)\mathbb{1}_{B^R_v})*_v
(|v|^{-\alpha}\mathbb{1}_{B^{2R}_v})\|_{L^{k'}_v}\\
&\le
\|f(x,\cdot)\mathbb{1}_{B^R_v}\|_{L^k_v}
\|f(x_*,\cdot)\|_{L^1_v}
\||v|^{-\alpha}\mathbb{1}_{B^{2R}_v}\|_{L^{k'}_v}.
\end{align*}
The last norm is finite because $\alpha k'<d$.
For $d\ge3$, this follows from $\alpha<2$;
for $d=2$, it follows from $k>\frac2{2-\alpha}$.
Integrating over $(t,x,x_*)$ implies \eqref{ineq-3}.

\item We split the collision integral into
$\{|v-v_*|>1\}$ and $\{|v-v_*|\le1\}$.
The first part satisfies
\begin{align*}
    \int_0^T\int_{\G\cap\{|v-v_*|>1\}}
    fg_*|v-v_*|^{-\alpha}\dd\eta\dd t
    \le
    \|f\|_{L^1_TL^1_{x,v}}
    \|g\|_{L^\infty_TL^1_{x,v}}.
\end{align*}
For the second part, let $\frac1{k-\varepsilon}+\frac1q=1$.
Hölder's and Young's convolution inequalities ensure
\begin{align*}
&\int_0^T\int_{\G\cap\{|v-v_*|\le1\}}
fg_*|v-v_*|^{-\alpha}\dd\eta\dd t\\
&\le
\||v|^{-\alpha}\|_{L^q(B^1)}
\|f\|_{L^1_TL^1_xL^{k-\varepsilon}_v}
\|g\|_{L^\infty_TL^1_{x,v}}.
\end{align*}
The condition
$\varepsilon<k-\frac{d}{d-\alpha}$
ensures that $\alpha q<d$.
Combining the two estimates proves \eqref{ineq-4}.

\end{enumerate}
\end{proof}

\section{Existence results}
\label{sec:existence}

In Section \ref{sec:weak-H}, we define weak and $\cH$-solutions and
discuss when their formulations are well-defined.
In Section \ref{sec:coe}, we establish a coercivity lemma.
In Section \ref{sec:sol}, we prove the existence of $\cH$-solutions.

\subsection{Weak solutions and \texorpdfstring{$\cH$}{TEXT}-solutions}
\label{sec:weak-H}

For solutions to the fuzzy Landau equation \eqref{Landau-fuz},
the following mass, momentum and energy conservation laws and
entropy identity hold at least formally
\begin{gather*}
\int_{\Do}(1,v,|v|^2)f_t\dd x\dd v
=\int_{\Do}(1,v,|v|^2)f_0\dd x\dd v
\quad\forall t\in[0,T],\\
\cH(f_T)-\cH(f_0)+\int_0^T\cD(f_t)\dd t=0.
\end{gather*}

The fuzzy Landau equation \eqref{Landau-fuz} can be written as
\begin{equation*}
\label{FL:a}
    \d_tf+v\cdot\nabla_xf
    =\sum_{i,j=1}^d\d_{v_i}\Big(
    \int_{\Do}\kappa(x-x_*)a_{ij}(v-v_*)
    \big(f_*\d_{v_j}f-f(\d_{v_j}f)_*\big)
    \dd x_*\dd v_*\Big),
\end{equation*}
where
\begin{align*}
    a_{ij}(v-v_*)
    =\Big(\delta_{ij}
    -\frac{(v-v_*)_i(v-v_*)_j}{|v-v_*|^2}\Big)
    A(|v-v_*|).
\end{align*}
We consider kernels $A$ satisfying \eqref{A-1} or \eqref{A-3}.

We define
\begin{align*}
    b_i(v-v_*)=\sum_{j=1}^d\d_j a_{ij}(v-v_*),\quad 
    c(v-v_*)=\sum_{i,j=1}^d\d^2_{ij}a_{ij}(v-v_*).
\end{align*}
We note that
\begin{equation}
\label{abc}
\begin{aligned}
    \operatorname{tr}(a)&=(d-1)A(|v-v_*|),\\
    b_i&=-(d-1)\frac{(v-v_*)_i}{|v-v_*|^2}A(|v-v_*|),\\
    c&=-(d-1)\Big(
    (d-2)\frac{A(|v-v_*|)}{|v-v_*|^2}
    +\frac{A'(|v-v_*|)}{|v-v_*|}
    \Big).
\end{aligned}
\end{equation}

We note that, for the power-law kernels satisfying \eqref{A-1}, we have
\begin{align*}
    |a(z)|&\lesssim |z|^{2+\gamma},\quad 
    |b(z)|\lesssim |z|^{1+\gamma},\quad 
    |c(z)|\lesssim |z|^\gamma.
\end{align*}
For kernels satisfying \eqref{A-3}, we additionally assume
\begin{equation}
\label{A3-2}
    A\in C^2(R_+;R_+)
    \quad\text{and}\quad
    |A'(|z|)|\lesssim |z|\langle z\rangle^\gamma
    \quad\forall z\in\R^d\setminus\{0\}.
\end{equation}
It follows from \eqref{abc} that
\begin{align*}
    |a(z)|&\lesssim |z|^2\langle z\rangle^\gamma,
    \quad
    |b(z)|\lesssim |z|\langle z\rangle^\gamma,
    \quad
    |c(z)|\lesssim\langle z\rangle^\gamma.
\end{align*}

We recall $\hf:=f*_x\kappa$. We define
\begin{equation*}
\label{bar:abc}
\begin{aligned}
    \overline a_{ij}(\hf)
    &=\hf*_va_{ij}
    =\int_{\Do}\kappa(x-x_*)a_{ij}(v-v_*)
    f(x_*,v_*)\dd x_*\dd v_*,\\
    \overline b_i(\hf)
    &=\hf*_vb_i
    =\int_{\Do}\kappa(x-x_*)b_i(v-v_*)
    f(x_*,v_*)\dd x_*\dd v_*,\\
    \overline c(\hf)
    &=\hf*_vc
    =\int_{\Do}\kappa(x-x_*)c(v-v_*)
    f(x_*,v_*)\dd x_*\dd v_*.
\end{aligned}
\end{equation*}
We write
$\overline a(\hf)=(\overline a_{ij}(\hf))\in\R^{d\times d}$ and
$\overline b(\hf)=(\overline b_i(\hf))^T\in\R^d$.
Notice that
\begin{align*}
    \sum_{j=1}^d\d_{v_j}\overline a_{ij}=\overline b_i,
    \quad
    \sum_{i=1}^d\d_{v_i}\overline b_i=\overline c.
\end{align*}
The fuzzy Landau equation then takes the divergence form
\begin{equation}
\label{FL:ab}
    \d_tf+v\cdot\nabla_xf
    =\div_v\big(\overline a\nabla_vf-\overline bf\big),
\end{equation}
and, under suitable regularity, the nondivergence form
\begin{equation*}
\label{FL:abc}
    \d_tf+v\cdot\nabla_xf
    =\sum_{i,j=1}^d\overline a_{ij}\d_{v_iv_j}f-\overline cf.
\end{equation*}

We introduce weak and $\cH$-solutions in
Definitions \ref{def:weak} and \ref{def:H}, respectively.

\begin{definition}[Weak solutions]
\label{def:weak}
Let $f_0\in L^1_{2,2}\cap\logl(\Do;\R_+)$.
We say that a nonnegative function
$f\in C([0,T];\cD'(\Do;\R_+))\cap L^\infty([0,T];L^1_{2,2}\cap \logl(\Do;\R_+))$ is a weak solution of
\eqref{Landau-fuz} if
\begin{equation}
\label{FL:weak:ab}
\begin{aligned}
&\int_{\Do}f_0\varphi(0)\dd x\dd v
+\int_0^T\int_{\Do}
f(\d_t\varphi+v\cdot\nabla_x\varphi)\dd x\dd v\dd t\\
={}&-\frac12\sum_{i,j=1}^d\int_0^T\int_{\G}
ff_*\kappa a_{ij}
\big(\d_{v_iv_j}\varphi+(\d_{v_iv_j}\varphi)_*\big)
\dd\eta\dd t\\
&-\sum_{i=1}^d\int_0^T\int_{\G}
ff_*\kappa b_i
\big(\d_{v_i}\varphi-(\d_{v_i}\varphi)_*\big)
\dd\eta\dd t
\end{aligned}
\end{equation}
for all $\varphi\in C^\infty_c([0,T)\times\Do)$,
with all terms finite.
\end{definition}

The right-hand side of \eqref{FL:weak:ab} can equivalently
be written as
\begin{equation*}
\label{RHS:bar-ab}
    -\int_0^T\int_{\Do}
    f\big(\overline a:\nabla_v^2\varphi
    +2\overline b\cdot\nabla_v\varphi\big)\dd x\dd v\dd t,
\end{equation*}
where, for $a_1,a_2\in\R^{d\times d}$, we write
$a_1:a_2=\sum_{i,j=1}^d(a_1)_{ij}(a_2)_{ij}$.

For $A(|v-v_*|)=|v-v_*|^{2+\gamma}$ with
$\gamma\in(-\gamma_d,-1)$, the mass, energy and entropy bounds
alone do not ensure the integrability required in
\eqref{FL:weak:ab}. In particular, additional control of
\begin{equation*}
\label{pb-1}
    \int_0^T\int_{\G}
    ff_*|v-v_*|^{1+\gamma}\dd\eta\dd t
\end{equation*}
is needed.

To treat singular kernels in the homogeneous Landau equation,
\textcite{Vil98b} introduced the notion of an $\cH$-solution.
We adopt an analogous formulation for the fuzzy Landau equation.
We recall the fuzzy Landau gradient introduced in
Part I \cite{DH25}:
\begin{equation*}
\label{gradient}
    \widetilde\nabla f(x,v)
    =\sqrt{A(|v-v_*|)}\Pi_{(v-v_*)^\perp}
    \big(\nabla_vf(x,v)-\nabla_{v_*}f(x_*,v_*)\big).
\end{equation*}
The corresponding divergence operator is defined by
\begin{equation}
\label{IP}
    \int_{\G}\widetilde\nabla g\cdot B\dd\eta
    =-\int_{\Do}g\,\widetilde\nabla\cdot B\dd x\dd v
\end{equation}
for smooth, sufficiently decaying functions
$B:\G\to\R^d$ and $g:\Do\to\R$.
Explicitly,
\begin{equation*}
\begin{aligned}
    \widetilde\nabla\cdot B
    =\nabla_v\cdot\Big(
    \int_{\Do}\sqrt{A(|v-v_*|)}\Pi_{(v-v_*)^\perp}
    \big(B(x,x_*,v,v_*)-B(x_*,x,v_*,v)\big)
    \dd x_*\dd v_*\Big).
\end{aligned}
\end{equation*}
Thus, the fuzzy Landau equation can be written as
\begin{equation}
\label{FL-H}
    (\d_t+v\cdot\nabla_x)f
    =\frac12\widetilde\nabla\cdot
    \big(\kappa ff_*\widetilde\nabla\log f\big).
\end{equation}
Based on \eqref{FL-H} and the integration by parts formula
\eqref{IP}, we follow \cite{Vil98b} to define an $\cH$-solution.

\begin{definition}[$\cH$-solution]\label{def:H}
Let $f_0\in L^1_{2,2}\cap\logl(\Do;\R_+)$.
We say that a nonnegative function
\begin{align*}
    f\in C([0,T];\cD'(\Do))
    \cap L^\infty([0,T];L^1_{2,2}(\Do)\cap\logl(\Do))
\end{align*}
is an $\cH$-solution of \eqref{Landau-fuz} if the entropy
dissipation is time-integrable,
\begin{align*}
    \int_0^T\cD(f_t)\dd t
    =\frac12\int_0^T\int_{\G}
    \kappa ff_*|\widetilde\nabla\log f|^2
    \dd\eta\dd t<+\infty,
\end{align*}
and the following identity holds:
\begin{equation}
\label{FL:H:ab}
\begin{aligned}
&\int_{\Do}f_0\varphi(0)\dd x\dd v
+\int_0^T\int_{\Do}
f(\d_t\varphi+v\cdot\nabla_x\varphi)\dd x\dd v\dd t\\
&\quad=\frac12\int_0^T\int_{\G}
\kappa ff_*\widetilde\nabla\varphi\cdot
\widetilde\nabla\log f\dd\eta\dd t
\end{aligned}
\end{equation}
for all $\varphi\in C^\infty_c([0,T)\times\Do)$,
with all terms finite.
\end{definition}

In the rest of this section, we discuss conditions under which
weak and $\cH$-solutions are well-defined and equivalent.

\begin{remark}\label{rmk-1}
For each test function $\varphi$, we choose $R>0$ sufficiently large
that
$\supp(\varphi)\subset[0,T)\times B^R_{x,v}$,
where $B_v^R=\{|v|\le R\}$.
In the periodic case, $B_x^R$ is replaced by $\T^d$.
\begin{itemize}

\item \underline{\it Well-definedness of the weak formulation in
Definition \ref{def:weak}.}
As a consequence of \eqref{abc}, we have
\begin{equation*}
\label{AB}
    |a(v-v_*)|\lesssim A(|v-v_*|),
    \quad
    |b(v-v_*)|\lesssim |v-v_*|^{-1}A(|v-v_*|).
\end{equation*}
By symmetry, it suffices to verify the finiteness of
\begin{equation}
\label{bd-1}
    \int_0^T
    \int_{\R^{2d}_{x_*,v_*}\times B^R_{x,v}}
    ff_*A(|v-v_*|)
    \max(1,|v-v_*|^{-1})\dd\eta\dd t.
\end{equation}

Under appropriate velocity-moment bounds, for example
$M_{2+\gamma_+}(f)\in L^1([0,T])$, \eqref{bd-1} is finite
when $A$ satisfies \eqref{A-1} with $\gamma\in[-1,1]$,
or \eqref{A-3} with $\gamma\in(-\infty,1]$.

When $A$ satisfies \eqref{A-1} with
$\gamma\in(-\min(3,d),-1)$, the Fisher-information bound
in Lemma \ref{lem:hydro}, together with
$\cD(f)\in L^1([0,T])$, provides the local integrability
needed to apply Lemma \ref{lem:sobo}-(6).
Indeed, $0<|1+\gamma|<2$.
Applying that estimate to the localised density controls the
near-diagonal contribution, while the moment bounds control
the contribution at infinity.

When $A$ satisfies \eqref{A-1} with
$\gamma\in(-\gamma_d,-3]$, finiteness follows by additionally
assuming
$f\in L^1([0,T];L^p_vL^1_x)$
for some $p>\frac{d}{d-|\gamma+1|}$,
and applying Lemma \ref{lem:sobo}-(2).

\smallskip

\item \underline{\it Well-definedness of the $\cH$-formulation in
Definition \ref{def:H}.}
By Cauchy--Schwarz and Young's inequalities, we have
\begin{align*}
&\Big|\int_0^T\int_{\G}
\kappa ff_*\widetilde\nabla\varphi\cdot
\widetilde\nabla\log f\dd\eta\dd t\Big|\\
&\lesssim_{\varphi,\kappa}
\int_0^T\cD(f_t)\dd t
+\int_0^T
\int_{\R^{2d}_{x_*,v_*}\times B^R_{x,v}}
ff_*A(|v-v_*|)\dd\eta\dd t.
\end{align*}
Since the entropy dissipation is time-integrable, it suffices
to verify the finiteness of
\begin{equation}
\label{bd-2}
    \int_0^T
    \int_{\R^{2d}_{x_*,v_*}\times B^R_{x,v}}
    ff_*A(|v-v_*|)\dd\eta\dd t.
\end{equation}
As in the weak formulation, \eqref{bd-2} is finite under
$M_{2+\gamma_+}(f)\in L^1([0,T])$
when $A$ satisfies \eqref{A-1} with $\gamma\in[-2,1]$,
or \eqref{A-3} with $\gamma\in(-\infty,1]$.

When $A$ satisfies \eqref{A-1} with
$\gamma\in(-\gamma_d,-2)$, the local weighted $L^p$-bound
in Lemma \ref{lem:sobo}-(3), together with
Lemma \ref{lem:sobo}-(6), controls the near-diagonal contribution.
Indeed, $\gamma>-\gamma_d\ge-4$ ensures that
$0<|2+\gamma|<2$.
The contribution at infinity is controlled by the moment bounds.

\smallskip

\item \underline{\it Equivalence of weak and $\cH$-formulations.}
Assume that $f$ satisfies the regularity requirements of both
definitions, $\cD(f)\in L^1([0,T])$, and both formulations
are well-defined.
Their equivalence follows by identifying the right-hand sides
of \eqref{FL:weak:ab} and \eqref{FL:H:ab}.
For smooth regularisations, this identity follows by integration
by parts. The passage to the limit is justified by the argument
for \eqref{proof-id} below.
For the corresponding homogeneous argument, we refer to \cite{Vil98}.

\end{itemize}
\end{remark}

\begin{remark}[A difference between homogeneous and fuzzy Landau equations]
\label{rmk:homo}
The $\cH$-formulation of the homogeneous Landau equation in
\cite{Vil98b} is more straightforward to justify when
$\gamma\in(-4,-2)$.

For test functions $\varphi=\varphi(v)\in C^\infty_c(\R^d)$,
we have
\begin{equation*}
    \widetilde\nabla_{\sf homo}\varphi(v)
    =|v-v_*|^{1+\frac{\gamma}{2}}
    \Pi_{(v-v_*)^\perp}
    \big(\nabla_v\varphi(v)-\nabla_{v_*}\varphi(v_*)\big).
\end{equation*}
Since
$|\nabla_v\varphi-\nabla_{v_*}\varphi_*|
\le|v-v_*|\|\varphi\|_{W^{2,\infty}}$,
we obtain
\begin{align*}
    |\widetilde\nabla_{\sf homo}\varphi|^2
    &\le |v-v_*|^{2+\gamma}
    |\nabla_v\varphi-\nabla_{v_*}\varphi_*|^2\\
    &\le |v-v_*|^{4+\gamma}
    \|\varphi\|_{W^{2,\infty}}^2.
\end{align*}
As $4+\gamma>0$, the test-function factor in the
Cauchy--Schwarz estimate has no singularity at $v=v_*$.

For the fuzzy Landau equation, general test functions depend
on both $x$ and $v$.
The difference
$\nabla_v\varphi(x,v)-\nabla_{v_*}\varphi(x_*,v_*)$
need not vanish when $v=v_*$ and $x\ne x_*$.
Hence, the same cancellation is not available.
\end{remark}

\subsection{Coercivity}\label{sec:coe}
In this section, we follow \cite{ALL15} to show the coercivity of $\overline a(x,v)$.
We recall the definition of $a$ and  $\overline a\in\R^{d\times d}$ 
\begin{align*}
  a_{ij}(z)=\Big(\delta_{ij}-\frac{z_iz_j}{|z|^2}\Big)|z|^{2+\gamma}\quad\text{and}\quad   \overline a_{ij}(x,v)=\hf*_va_{ij}.
\end{align*}

Let $f\in L^\infty([0,T];L^1_{2,2}(\Do))$ be a weak solution of the fuzzy Landau equation \eqref{Landau-fuz} such that  
\begin{align*}
   \|f\|_{L^1(\Do)}=1,\quad \|\langle v\rangle^2f\|_{L^1_{x,v}}\le E_0\quad\text{and}\quad \|f(\log f)_+\|_{L^1_{x,v}}\le H_0.
\end{align*}

We show the following coercivity lemma.
\begin{lemma}
\label{lem:coer}
Let $A$ and $\kappa$ satisfy Assumption \ref{ASS:kernel}.
Let $f\in L^1_{2,2}(\Do)\cap\logl(\Do;\R_+)$ satisfy
\begin{align*}
    \|\langle v\rangle^2f\|_{L^1_{x,v}}\le E_0,
    \quad
    \|f(\log f)_+\|_{L^1_{x,v}}\le H_0.
\end{align*}
The coercivity of $\overline a(x,v)$ depends on the spatial
kernel $\kappa$ as follows:
\begin{itemize}
\item \underline{Spatial kernels bounded away from zero.}
If $\kappa\sim 1$ is given as in \eqref{k-1} or \eqref{k-3}, then we have
\begin{equation*}
\label{coe-2}
    \xi^T\overline a(x,v)\xi
    \ge C_{coe}\langle v\rangle^\gamma|\xi|^2
    \quad\forall x,v,\xi.
\end{equation*}

\item \underline{Spatial kernels vanishing at infinity.}
If $\kappa$ takes the form of Gaussian on $\R^d$ as in \eqref{k-2}, then we have
\begin{equation*}
\label{coe-1}
    \xi^T\overline a(x,v)\xi
    \ge C_{coe}\langle v\rangle^\gamma|\xi|^2
    \exp\big(-C\exp(C\langle x\rangle)\big)
    \quad\forall x,v,\xi.
\end{equation*}
\end{itemize}
The constants $C_{coe},C>0$ are independent of $x,v,\xi$
and may depend on
$d,\gamma,\kappa,E_0,H_0$ and $\|f\|_{L^1_{x,v}}$.
In the whole space case, they may additionally depend on
$\|f\|_{L^1_{2,0}}$.
\end{lemma}
\begin{proof}
The proof follows \cite{ALL15} with an appropriate modification for the spatial variable $x\in\R^d$.
We define $m_0=m(\hf)>0$. The positive lower bound for $m(\hf)$ 
\begin{gather*}
  m(\hf)\gtrsim_\kappa\|f\|_{L^1_{x,v}}\quad \text{$\kappa$ bounded from below in \eqref{k-1} or \eqref{k-3}}\\
   m(\hf)\gtrsim
    \exp\big(- C\langle x \rangle \big)\quad \text{$\kappa$ vanishing at infinity in \eqref{k-2}},
\end{gather*}
and the other hydrodynamic bounds for $\hf$ given in Lemma \ref{lem:hydro} will be used repeatedly below.

We recall the definition $H(\hf)=\int_v \hf(\log \hf)_+$.
We choose $L>1$ large enough that
\begin{equation*}
    \int_{B^L_v}\hf\dd v\ge m_0/2.
\end{equation*}
For arbitrarily fixed $\theta\in(0,1)$ and $|\xi|=1$, we define the following cone centred at $v\in\R^d$ as follows
\begin{align*}
    \Omega_{\theta,\xi,v}\defeq
    \Big\{v_*\in\R^d\setminus\{v\}\Bigm\vert
    \Big|\frac{v-v_*}{|v-v_*|}\cdot\xi\Big|\ge\cos\theta\Big\}.
\end{align*}
For all $\xi\in\R^d\setminus\{0\}$, we write $\hat\xi=\frac{\xi}{|\xi|}$ and, for notational convenience, we write $\Omega_{\theta,\xi,v}=\Omega_{\theta,\hat\xi,v}$.
On the one hand, we have
\begin{align*}
    \big|B^L_{v_*}\cap\Omega_{\theta,\hat\xi,v}\big|
    \le2|S^{d-1}|L(L+|v|)^{d-1}\tan^{d-1}\theta.
\end{align*}
On the other hand, for $v\in\R^d$ and $v_*\in\Omega^c_{\theta,\xi,v}$ with $v_*\ne v$, we have
\begin{align*}
    &\xi^Ta(v-v_*)\xi\\
    ={}&A(|v-v_*|)\sum_{i,j=1}^d
    \Big(\delta_{ij}-\frac{(v-v_*)_i(v-v_*)_j}{|v-v_*|^2}\Big)\xi_i\xi_j\\
    ={}&A(|v-v_*|)
    \Big(|\xi|^2-\frac{|(v-v_*)\cdot\xi|^2}{|v-v_*|^2}\Big)\\
    ={}&A(|v-v_*|)
    \Big(1-\frac{|(v-v_*)\cdot\hat\xi|^2}{|v-v_*|^2}\Big)|\xi|^2\\
    \ge{}&A(|v-v_*|)\cos^2\theta\tan^2\theta|\xi|^2.
\end{align*}
Hence,  we have
\begin{align*}
    \xi^T\overline a(x,v)\xi
    &\gtrsim\tan^2\theta|\xi|^2
    \int_{B^L_{v_*}\cap\Omega_{\theta,\hat\xi,v}^c}
    A(|v-v_*|)\hf(x,v_*)\dd v_*.
\end{align*}

\smallskip

We split the domain $\R^d_v=B^{2L}_v\cup(B^{2L}_v)^c$.
\begin{itemize}
    \item {\bf The case of $|v|>2L$.}
    Splitting into $0\le\hf\le K$ and $\hf>K$, we obtain
    \begin{equation}
        \label{mass-1}
    \begin{aligned}
        \int_{B^L_{v_*}\cap\Omega_{\theta,\hat\xi,v}^c}
        \hf(x,v_*)\dd v_*
        &\ge\int_{\R^d}\hf(x,v_*)\dd v_*
        -\int_{(B^L_{v_*})^c}\hf(x,v_*)\dd v_*\\
        &\quad-\int_{B^L_{v_*}\cap\Omega_{\theta,\hat\xi,v}}
        \hf(x,v_*)\dd v_*\\
        &\ge\frac{m_0}{2}-\delta K-\frac{H(\hf)}{\log K}
        \ge\frac{m_0}{4},
    \end{aligned}
    \end{equation}
    where we have
    $\big|B^L_{v_*}\cap\Omega_{\theta,\hat\xi,v}\big|\le\delta$, and we choose $K=\exp\Big(1+\frac{8H(\hf)}{m_0}\Big)$,
    $\delta=\frac{m_0}{8K}$, and
    \begin{equation}
        \label{tan}
        \tan\theta
        =\min\left\{
        \frac12,\,
        \Big(
        \frac{m_0}{16K|S^{d-1}|L(L+|v|)^{d-1}}
        \Big)^{\frac{1}{d-1}}
        \right\}.
    \end{equation}

    We then show that
    \begin{equation*}
        A(|v-v_*|)\gtrsim\langle v\rangle^{2+\gamma}
        \quad\forall v_*\in B^L_{v_*}.
    \end{equation*}
    Since $|v_*|\le L$, we have
    $\frac{|v|}{2}\le|v-v_*|\le\frac{3|v|}{2}$.

    We first consider $A(|v-v_*|)=|v-v_*|^{2+\gamma}$ satisfying \eqref{A-1}.
    For $\gamma\in[-2,1]$, we have
    \begin{align*}
        |v-v_*|^{2+\gamma}
        \ge2^{-(2+\gamma)}|v|^{2+\gamma}
        \gtrsim\langle v\rangle^{2+\gamma}.
    \end{align*}
    For $\gamma\in(-\gamma_d,-2)$, we have
    \begin{align*}
        |v-v_*|^{2+\gamma}
        \ge\Big(\frac{3|v|}{2}\Big)^{2+\gamma}
        \gtrsim\langle v\rangle^{2+\gamma}.
    \end{align*}

The case of $A$ satisfying \eqref{A-3} can be treated analogously to the hard potential cases.

    It follows from \eqref{tan} that
    \begin{align*}
        \tan^2\theta
        \gtrsim
        \min\Big\{1,\,
        m_0^{\frac{2}{d-1}}
        \exp\big(-CH(\hf)/m_0\big)\Big\}
        \langle v\rangle^{-2}.
    \end{align*}
    Combining these estimates, we obtain
    \begin{align*}
        \xi^T\overline a(x,v)\xi
        &\gtrsim
        m_0\min\Big\{1,\,
        m_0^{\frac{2}{d-1}}
        \exp\big(-CH(\hf)/m_0\big)\Big\}
        \langle v\rangle^\gamma|\xi|^2.
    \end{align*}
    We obtain the coercivity constant by applying Lemma \ref{lem:hydro}.

\smallskip

    \item {\bf The case of $|v|\le2L$.}
    We choose
    \begin{align*}
        K=\exp\Big(1+\frac{16H(\hf)}{m_0}\Big),
        \quad
        \delta=\min\Big\{\frac14,\frac{m_0}{16K}\Big\},
    \end{align*}
    and
    \begin{equation}
        \label{tan-2}
        \tan\theta
        =\min\left\{
        \frac12,\,
        \Big(
        \frac{\delta}
        {2|S^{d-1}|L(L+|v|)^{d-1}}
        \Big)^{\frac{1}{d-1}}
        \right\}.
    \end{equation}
    Then we have
    \begin{align*}
        \big|B^L_{v_*}\cap\Omega_{\theta,\hat\xi,v}\big|
        \le\delta,
        \quad
        \big|\{|v-v_*|\le\delta\}\big|
        \le2^d\delta^d\le\delta.
    \end{align*}
    Following \eqref{mass-1}, we obtain
    \begin{equation}
        \label{mass-2}
    \begin{aligned}
        &\int_{B^L_{v_*}\cap\Omega_{\theta,\hat\xi,v}^c}
        \hf(x,v_*)\mathbb{1}_{\{|v-v_*|\ge\delta\}}\dd v_*\\
        &\ge
        \int_{B^L_{v_*}\cap\Omega_{\theta,\hat\xi,v}^c}
        \hf(x,v_*)\dd v_*
        -\int_{|v-v_*|\le\delta}\hf(x,v_*)\dd v_*\\
        &\ge\frac{m_0}{2}-2\delta K
        -2\frac{H(\hf)}{\log K}
        \ge\frac{m_0}{4}.
    \end{aligned}
    \end{equation}

    We first consider $A$ satisfying \eqref{A-1}.
    Since $|v_*|\le L$, we have $|v-v_*|\le3L$.
    On the set $\{\delta\le|v-v_*|\le3L\}$, we have
    \begin{align*}
        A(|v-v_*|)=|v-v_*|^{2+\gamma}
        \ge
        \begin{cases}
            \delta^{2+\gamma},&\gamma\ge-2,\\
            (3L)^{2+\gamma},&\gamma<-2.
        \end{cases}
    \end{align*}
    Hence, we have
    \begin{align*}
        A(|v-v_*|)\gtrsim\delta^{\max\{2+\gamma,0\}}.
    \end{align*}
    Since $|v|\le2L$, it follows from \eqref{tan-2} that $\tan^2\theta\gtrsim\delta^{\frac{2}{d-1}}$.

    Combining these estimates with \eqref{mass-2}, we obtain
    \begin{gather*}
        \xi^T\overline a(x,v)\xi
        \gtrsim
        m_0\delta^{\frac{2}{d-1}+\max\{2+\gamma,0\}}|\xi|^2
    \end{gather*}

    For kernels satisfying \eqref{A-3}, on
$\{\delta\le|v-v_*|\le3L\}$, we have
\begin{align*}
    A(|v-v_*|)
    &\gtrsim |v-v_*|^2\langle v-v_*\rangle^\gamma\ge \delta^2\langle3L\rangle^{\min\{\gamma,0\}}.
\end{align*}
Combining this estimate with \eqref{tan-2} and
\eqref{mass-2}, we obtain
\begin{align*}
    \xi^T\overline a(x,v)\xi
    \gtrsim_{L,\gamma}
    m_0\delta^{2+\frac{2}{d-1}}|\xi|^2.
\end{align*} 
Hence, we have, in both cases
\begin{gather*}
        \xi^T\overline a(x,v)\xi
        \gtrsim
        m_0\min\left\{
        1,\,
        m_0^{\frac{2}{d-1}+\max\{2+\gamma,0\}}
        \exp\big(-CH(\hf)/m_0\big)
        \right\}|\xi|^2.
    \end{gather*}
     We obtain the coercivity constant by applying Lemma \ref{lem:hydro}.
\end{itemize}

\end{proof}

\subsection{Existence of \texorpdfstring{$\cH$}{TEXT}-solutions}\label{sec:sol}

In Remark \ref{rmk-1}, we discuss the equivalence of weak and $\cH$-solutions. In this subsection, we show the existence of $\cH$-solutions.
\begin{theorem}\label{thm:exist}
  Let the kernels $A$ and $\kappa$ satisfy  Assumption \ref{ASS:kernel}. Let $f_0\in L^1_{2,2}\cap \logl(\Do;\R_+)$. When $\gamma>0$, we additionally assume $f_0\in L^1_{2,2+\delta}$ for some $\delta>0$. 
  
  There exists an $\cH$-solution $f\in C([0,T];\cD'(\Do;\R_+))\cap L^\infty([0,T];L^1_{2,2}(\Do;\R_+))$ of the fuzzy Landau equation \eqref{Landau-fuz} with initial value $f_0$.
In addition, in the hard potential case $\gamma>0$:
  \begin{itemize}
      \item \underline{Spatial kernels bounded away from zero.} If $\kappa\sim 1$ is given as in \eqref{k-1} or \eqref{k-3}, then we have  
      \begin{align*}
       \langle v\rangle^{2+\delta}f\in L^\infty\big([0,T];L^1(\Do)\big)\quad\text{and}\quad \langle v\rangle^{2+\gamma+\delta}f\in L^1\big([0,T]\times \Do\big).
       \end{align*}

\item \underline{Spatial kernels vanishing at infinity.}
If $\kappa$ takes the form of Gaussian on $\R^d$ as in \eqref{k-2}, then we have  
      \begin{align*}
       \langle v\rangle^{2+\delta}f\in L^\infty\big([0,T];L^1(\Do)\big).
       \end{align*}
  \end{itemize}
  Moreover, the mass and momentum conservation laws hold:
  \begin{equation*}
    \int_{\Do}(1,v)f_t\dd x\dd v=\int_{\Do}(1,v)f_0\dd x\dd v\quad\forall t\in[0,T].
  \end{equation*}
The following energy, spatial second moment and entropy inequalities hold for all $t\in[0,T]$:
\begin{equation*}
\label{ineq:thm}
   \begin{aligned}
&\int_{\Do}|v|^2f_t\dd x\dd v\le\int_{\Do}|v|^2f_0\dd x\dd v,\\
&\int_{\Do}|x|^2f_t\dd x\dd v\le C(T)\int_{\Do}(|x|^2+|v|^2)f_0\dd x\dd v,\\
&\cH(f_t)-\cH(f_0)+\int_0^t\cD(f_s)\dd s\le0.
   \end{aligned}
\end{equation*}
\end{theorem}

\begin{proof}
    We first approximate the initial data by $\tilde f^n_0\in C^\infty_c(\Do;\R_+)$ such that $0\le \tilde f^n_0\le n$,
\begin{align*}
    \int_{\Do}|\tilde f^n_0-f_0|(1+|x|^2+|v|^2)\dd x\dd v\to0\quad\text{as}\quad n\to\infty,
\end{align*}
and
\begin{align*}
    &\int_{\Do}\tilde f^n_0\big(1+|x|^2+|v|^2+|\log \tilde f^n_0|\big)\dd x\dd v\le C
\end{align*}
uniformly in $n$.

For $x\in\R^d$, we define
\begin{equation*}
\label{initial:app}
    f^n_0:=\tilde f^n_0+\frac{1}{n}\exp\big({-\frac{|v|^2}{2}}-\frac{|x|^2}{2}\big).
\end{equation*}
For $x\in\T^d$, we take $f^n_0=\tilde f^n_0+\frac{1}{n}\exp(-|v|^2/2)$. We will only show the proof for $x\in\R^d$, and the torus case follows analogously, with the spatial moment arguments omitted.

Let $\chi=\chi(r)$ be a smooth nonincreasing cut-off function such that $0\le \chi\le 1$, $\chi(r)=1$ for $r\in[0,1]$ and $\chi(r)=0$ for $r\in[2,\infty)$. Let $\chi_n(r)=\chi(r/n)$.
For $A(z)=|z|^{2+\gamma}$ with $\gamma\in[0,1]$ or $A(|z|)\sim |z|^2\langle v\rangle^\gamma$ with $\gamma\in(-\infty,1]$, we take $A_n(|z|):=A(|z|)\chi_n(|z|)$. For $A(z)=|z|^{2+\gamma}$ with $\gamma\in(-\gamma_d,0)$, we take $A_n(|z|):=A(|z|)\big(\chi_n-\chi_{n^{-1}}\big)(|z|)$.

Then $A_n$ is smooth and compactly supported (and vanishes near the origin in the soft potential cases) such that
\begin{gather*}
0\le A_n\le A_{n+1}\le A \quad \forall\,n\in\mathbb{N},\\
 A_n(|v-v_*|)\to A(|v-v_*|)\quad\text{as}\quad n\to\infty.
\end{gather*}

The corresponding coefficients are given by  
\begin{align*}
 &a_{n,ij}(v-v_*):=\Big(\delta_{ij}-\frac{(v-v_*)_i(v-v_*)_j}{|v-v_*|^2}\Big)A_n(|v-v_*|),\\
 &b_{n,i}(v-v_*):=\sum_{j=1}^d\d_j a_{n,ij}(v-v_*),\quad
 c_n(v-v_*):=\sum_{i,j=1}^d\d^2_{ij}a_{n,ij}(v-v_*),
\end{align*}
For the power-law kinetic kernels given by \eqref{A-1}, uniformly in $n\in\mathbb{N}$, we have 
\begin{align*}
\label{bdd:ab-n}
 |a_n|\lesssim |v-v_*|^{\gamma+2}
 ,\quad
 |b_n|\lesssim |v-v_*|^{\gamma+1},\quad |c_n(v-v_*)|\lesssim |v-v_*|^\gamma.
\end{align*}
For the kinetic kernels given by \eqref{A-3}, similarly, we have 
\begin{equation*}
\label{bdd:ab-n-2}    
\begin{gathered}
 |a_n(v-v_*)|\lesssim |v-v_*|^2\langle v-v_*\rangle^\gamma,\quad   |b_n(v-v_*)|
    \lesssim |v-v_*|\langle v-v_*\rangle^\gamma,\\
    \text{and}\quad 
    |c_n(v-v_*)|
    \lesssim\langle v-v_*\rangle^\gamma.
\end{gathered}
\end{equation*}

We consider the following Cauchy problem
\begin{equation}
\label{FL:app}
    \left\{
    \begin{aligned}
    &\d_t f^n+v\cdot\nabla_x f^n=Q^n(f^n,f^n)+\frac{1}{n}\Delta_{v} f^n\\
    &f^n|_{t=0}=f^n_0,
    \end{aligned}
    \right.
\end{equation}
where $Q^n(f,f)=\div_v(\overline a^n\nabla_v f-\overline b^nf)$.
For each fixed $n$, a standard fixed-point argument and the ellipticity in $v$ provided by $\frac{1}{n}\Delta_v$ yield a unique solution $f^n\in C^1([0,T];C^\infty(\Do))$ of \eqref{FL:app}; see, for example, \cite{Vil96}. The maximum principle \cite{AB90} gives
\begin{equation*}
    f^n\ge C_ne^{-a(|x|^2+|v|^2)}.
\end{equation*}

Then we have 
$\|f^n_t\|_{L^1(\Do)}=\|f^n_0\|_{L^1(\Do)}$ and
\begin{align*}
   &\frac{d}{dt}\int_{\Do} f^n|v|^2\dd x\dd v=\frac{2d}{n}\int_{\Do}f^n\dd x\dd v,\\
   &\frac{d}{dt}\int_{\Do} f^n|x|^2\dd x\dd v=2\int_{\Do}x\cdot vf^n\dd x\dd v,\\
   &\frac{d}{dt}\cH(f^n)=-\cD_n(f^n)-\frac{1}{n}\int_{\Do}\frac{|\nabla_v f^n|^2}{f^n}\dd x\dd v\le0,
\end{align*}
where $\cD_n(f):=\frac12 \int_{\G}\kappa ff_*|\widetilde \nabla_n\log f|^2\dd\eta$ and $\widetilde\nabla_n f:=\sqrt{A_n}\Pi_{(v-v_*)^\perp}(\nabla_vf-\nabla_{v_*}f_*)$.
Integrating in time, for all $t\in[0,T]$, we have 
\begin{gather}
    \int_{\Do}f^n_t|v|^2\dd x\dd v
    =\int_{\Do}f^n_0|v|^2\dd x\dd v
    +\frac{2dt}{n}\|f_0^n\|_{L^1(\Do)},\label{app:consv-law}\\
    \int_{\Do}f^n_t|x|^2\dd x\dd v
    \le e^{4T}\Big(\int_{\Do}f^n_0|x|^2\dd x\dd v
    +\int_0^T\int_{\Do}f_s^n|v|^2\dd x\dd v\dd s\Big),\label{app:consv-law-2}\\
    \cH(f^n_t)-\cH(f^n_0)+\int_0^t\cD_n(f^n_s)\dd s\le0.\label{app:consv-law-3}
\end{gather}
For the Maxwiellien and hard potential cases, we will use the following claim, proved at the end of this section.
\begin{claim}
\label{sol:moments}
    \begin{itemize}
  \item Let $\gamma=0$. There exists a nonnegative convex function
$\Psi\in C^2([0,\infty))$, depending on $f_0$, such that
$$
\Psi(0)=\Psi'(0)=0,\quad
0\le\Psi''\le1,\quad
\lim_{r\to\infty}\frac{\Psi(r)}{r}=+\infty,
$$
and
$$
\int_{\Do}\Psi(|v|^2)f_0(x,v)\dd x\dd v<\infty.
$$
Moreover, we have
\begin{equation}
    \label{app:consv-law:max}
     \sup_n\sup_{0\leq t\leq T}
    \int_{\Do}\Psi(|v|^2)f^n_t(x,v)\dd x\dd v<\infty.
\end{equation}

        \item Let $\gamma\in(0,1]$. We have
        \begin{equation}
        \label{app-hard}
        \begin{aligned}
        &\sup_n\Big(\sup_{0\le t\le T}\int_{\Do}\langle v\rangle^{2+\delta} f^n_t+\int_0^t\int_{\R^{4d}}\langle v\rangle^{2+\gamma+\delta} f^nf^n_*\kappa\Big)\le C.    
        \end{aligned}
        \end{equation}
    \end{itemize}
\end{claim}
When $\kappa\gtrsim 1$, as in \eqref{k-1} and \eqref{k-3}, \eqref{app-hard} implies
\begin{align*}
\sup_n\int_0^t\int_{\Do}\langle v\rangle^{2+\gamma+\delta} f^n\le C.    
        \end{align*}
To control the negative part of the entropy, we follow \cite{DL89} and write $(\log f)^-=\max(0,-\log f)$. For $\varepsilon\in(0,\frac{1}{d+1})$, the bound $0\le(\log f)^-\lesssim f^{-\varepsilon}$ and Hölder's inequality imply
        \begin{equation}
        \label{log-}
            \int_{\Do}f(\log f)^-\dd x\dd v\lesssim\|f\|_{L^1_{2,2}(\Do)}^{1-\varepsilon}\|(\langle x\rangle +\langle v\rangle)^{-\frac{2(1-\varepsilon)}{\varepsilon}}\|_{L^1(\Do)}^\varepsilon,
        \end{equation}
       since $\frac{2(1-\varepsilon)}{\varepsilon}>2d$.
       Together with $|f\log f|=f\log f+2f(\log f)^-$ and the moment bounds, \eqref{app:consv-law-3} and \eqref{log-} imply
        \begin{equation*}
        \label{LlogL}
\|f^n_t\log f^n_t\|_{L^1(\Do)}+\int_0^t\cD_n(f^n_s)\dd s\le \cH(f^n_0)+C(T)\|f^n_0\|_{L^1_{2,2}(\Do)}^{1-\varepsilon},
        \end{equation*}
where $C(T)$ is independent of $n$. Thus, $(f^n)$ is bounded in $L^1_{2,2}\cap\logl(\Do)$ uniformly in $n$ and $t\in[0,T]$, and $(\cD_n(f^n))$ is bounded in $L^1([0,T])$.

The solutions of \eqref{FL:app} satisfy the following weak formulation 
\begin{equation}
    \label{weak:app-1}
        \begin{aligned}
       &\int_{\Do}f^n_0\varphi(0)\dd x\dd v+\int_{0}^T\int_{\Do}f^n(\d_t\varphi+v\cdot\nabla_x\varphi+n^{-1}\Delta_v \varphi)\dd x\dd v\dd t\\
&=-\frac12\int_0^T\int_{\G}A_n\kappa f^nf^n_* \Phi\dd\eta\dd t,
        \end{aligned}
\end{equation}
for all $\varphi\in C^\infty_c([0,T)\times\Do)$, where $\Phi:=(\nabla_v-\nabla_{v_*})\cdot \big(\Pi_{(v-v_*)^\perp}(\nabla_v\varphi-\nabla_{v_*}\varphi_*)\big)$. 

For any $t\in[0,T]$, and some  $R$ sufficiently large, we have 
\begin{equation*}
\begin{gathered}
\supp\Phi\subset \big(\R^{2d}_{x,v}\times B^R_{x_*,v_*}\big)\cup \big(B^R_{x,v}\times\R^{2d}_{x_*,v_*}\big)\\
\text{and}\quad \big|\Phi\big|
    \lesssim \|\nabla_v^2\varphi\|_{L^\infty_{x,v}}
    +|v-v_*|^{-1}\|\nabla_v\varphi\|_{L^\infty_{x,v}}.
\end{gathered}
\end{equation*}

We now pass to the limit in \eqref{weak:app-1} and identify the limit as an $\cH$-solution satisfying \eqref{FL:H:ab}.

\medskip

We distinguish two cases: Case \ref{case-1} covers \eqref{A-1} with $\gamma\in[-1,1]$ and \eqref{A-3} with $\gamma\in(-\infty,1]$; Case \ref{case-2} covers \eqref{A-1} with $\gamma\in(-\gamma_d,-1)$.

\begin{enumerate}[\text{Case} $1$]

\item \label{case-1}
 The uniform mass, entropy and moment bounds and the Dunford--Pettis theorem ensure, up to a subsequence,
\begin{align*}
f^n\rightharpoonup f\quad\text{in}\quad
L^p\big([0,T];L^1(\Do)\big)\quad\forall\,p\in[1,\infty).
\end{align*}
For every $\varphi\in C^\infty_c(\Do)$ and $0\le s\le t\le T$, we have
\begin{equation}
\label{continu-time}
\begin{aligned}
    \Big|\int_{\Do}(f^n_t-f^n_s)\varphi\dd x\dd v\Big|
    &\le \int_s^t\Big|\int_{\Do}f^n_r
    (v\cdot\nabla_x\varphi+n^{-1}\Delta_v\varphi)
    \dd x\dd v\Big|\dd r\\
    &\quad+\Big|\int_s^t\int_{\Do}
    Q^n(f_r^n,f_r^n)\varphi\dd x\dd v\dd r\Big|.
\end{aligned}
\end{equation}
The first term is bounded by $C|t-s|$, uniformly in $n$, while the collision term satisfies
\begin{align*}
\Big|\int_s^t\int_{\Do}Q^n(f_r^n,f_r^n)\varphi\dd x\dd v\dd r\Big|
\le C\int_s^t\int_{\G}
\langle v\rangle^{2+\gamma_+}\kappa f^nf_*^n,
\end{align*}
where $C$ may depend on $\varphi$ but is independent of $n$.
For $\gamma\le0$, the second-moment bound ensures
\begin{align*}
    \Big|\int_s^t\int_{\G}
\langle v\rangle^{2+\gamma_+}\kappa f^nf_*^n\Big|\lesssim_T |s-t|\|\langle v\rangle^2f^n\|_{L^1}^2.
\end{align*}
For $\gamma>0$, interpolation between \eqref{app-hard} and the second-moment bound, followed by Hölder's inequality, implies that 
\begin{align*}
\Big|\int_s^t\int_{\G}
\langle v\rangle^{2+\gamma_+}\kappa f^nf_*^n\Big|
\lesssim_T |t-s|^{\frac{\delta}{\gamma+\delta}}\|\langle v\rangle^2f^n\|_{L^1}^{\frac{2\delta}{\gamma+\delta}}.
\end{align*}
The refined Ascoli--Arzelà theorem \cite[Proposition 3.3.1]{AGS08} therefore ensures, up to a subsequence, $f^n_t\rightharpoonup f_t$ in $L^1(\Do)$ for every $t\in[0,T]$, with $f$ weakly continuous in $L^1(\Do)$. Hence, we have 
\begin{align*}
 f^nf^n_*\rightharpoonup ff_*\quad\text{in} \quad L^1([0,T]\times\G).  
\end{align*}

For the collision term in \eqref{weak:app-1}, we first pass to the limit on $B^M_{v,v_*}$
\begin{gather*}
   \int_0^T\int_{\R^{2d}_{x,x_*}\times B^M_{v,v_*}}
   f^nf^n_*\kappa A_n\Phi\dd\eta\dd t\\
   \to{}
   \int_0^T\int_{\R^{2d}_{x,x_*}\times B^M_{v,v_*}}
   ff_*\kappa A\Phi\dd\eta\dd t\quad \text{as $n\to\infty$},
\end{gather*}
where we use $|A_n\Phi|\lesssim\langle v-v_*\rangle^{\gamma+2}\lesssim_M1$ on $B^M_{v,v_*}$.

For $M$ sufficiently large, the compact velocity support of $\varphi$ and symmetry ensure that
\begin{align*}
&\Big|\int_0^T\int_{\G\cap\{\max(|v|,|v_*|)>M\}}
f^nf^n_*\kappa A_n\Phi\dd\eta\dd t\Big|\\
\lesssim&{}
\int_0^T\int_{\G\cap\{|v|>M\}}
\langle v\rangle^{2+\gamma}f^nf^n_*\kappa\dd\eta\dd t.
\end{align*}
For $\gamma<0$, the uniform $L^\infty_tL^1_{0,2}$ bound gives, with $\alpha=|\gamma|$, we have
\begin{align*}
&\int_0^T\int_{\Do\cap\{|v|>M\}}
\langle v\rangle^{2+\gamma}f^n\dd x\dd v\dd t\\
\lesssim&{} M^{-\alpha}
\int_0^T\int_{\Do}\langle v\rangle^2f^n\dd x\dd v\dd t
\to0
\end{align*}
uniformly in $n$ as $M\to\infty$.

For $\gamma=0$, Claim \ref{sol:moments} and the superlinearity of $\Psi$ give uniform control of the second-moment tails. For $\gamma>0$, \eqref{app-hard} ensures that
\begin{align*}
&\int_0^T\int_{\G\cap\{|v|>m\}}
\kappa f^nf^n_*\langle v\rangle^{2+\gamma}
\dd\eta\dd t\\
\le&{}
M^{-\delta}\int_0^T\int_{\G}
\kappa f^nf^n_*\langle v\rangle^{2+\gamma+\delta}
\dd\eta\dd t\to0
\end{align*}
as $M\to\infty$ uniformly in $n$.
These tail bounds complete the passage to the limit in \eqref{weak:app-1}.

\smallskip

\item \label{case-2}
The same bounds give weak compactness as in Case \ref{case-1}. To control the singular collision term, we use
\begin{equation}
\label{D-n}
\sup_n \int_0^T\cD_n(f^n)\dd t\le C.
\end{equation}

Following \cite{Vil98b}, fix $0<\delta<1$ and a radial $\psi\in C^\infty_c(\R^d;[0,1])$ with $\supp(\psi)\subset B_\delta(0)$ and $\psi\equiv1$ on $B_{\delta/2}(0)$. We decompose
\begin{align*}
   |v-v_*|^{2+\gamma}
   &=
   |v-v_*|^{2+\gamma}(1-\psi(v-v_*))
   +|v-v_*|^{2+\gamma}\psi(v-v_*)\\
   &=:\Psi^1(|v-v_*|)+\Psi^2(|v-v_*|).
\end{align*}
For the near-diagonal part $\Psi^2$, integration by parts and the radiality of $\psi$ give
\begin{align*}
&\int_0^T\int_{\G}
\kappa A_n\psi(v-v_*)f^nf^n_*\Phi\dd\eta\dd t\\
=&{}
-\int_0^T\int_{\G}
\kappa\psi(v-v_*)f^nf^n_*
\widetilde\nabla_n\log f^n\cdot\widetilde\nabla_n\varphi\dd\eta\dd t.
\end{align*}
By symmetry and the compact support of $\varphi$, the final integral below is restricted to $\R^d_{x_*}\times B^R_{x,v,v_*}$, with $R$ independent of $\delta\in(0,1)$.
Take $0<\varepsilon<2+\gamma$ for $\gamma\in(-2,-1)$, and $0<\varepsilon<\min(4+\gamma,1)$ for $\gamma\in(-\gamma_d,-2]$. By Cauchy--Schwarz and \eqref{D-n}, we have 
\begin{equation}
\label{singu}
\begin{aligned}
&\Big|\int_0^T\int_{\G}
\kappa A_n\psi(v-v_*)f^nf^n_*\Phi\dd\eta\dd t\Big|\\
\le&{}
\Big(2\int_0^T\cD_n(f^n)\dd t\Big)^{\frac12}
\Big(\int_0^T\int_{\G}
\kappa\psi(v-v_*)^2f^nf^n_*
|\widetilde\nabla_n\varphi|^2\dd\eta\dd t\Big)^{\frac12}\\
\lesssim&{}
\delta^{\frac{\varepsilon}{2}}
\Big(\int_0^T\int_{\R^d_{x_*}\times B^R_{x,v,v_*}}
f^nf^n_*|v-v_*|^{2+\gamma-\varepsilon}
\dd\eta\dd t\Big)^{\frac12}.
\end{aligned}
\end{equation}
For $\gamma\in(-2,-1)$, we have
\begin{equation*}
\begin{aligned}
\int_0^T\int_{\R^{4d}}
f^nf^n_*|v-v_*|^{2+\gamma-\varepsilon}\dd\eta\dd t
\lesssim_T \|f^n\|_{L^\infty_tL^1_{0,2}}^2.
\end{aligned}
\end{equation*}
For $\gamma\in(-\gamma_d,-2]$, we have
\begin{equation*}
\begin{aligned}
&\int_0^T\int_{\R^d_{x_*}\times B^R_{x,v,v_*}}
f^nf^n_*|v-v_*|^{2+\gamma-\varepsilon}\dd\eta\dd t\\
\lesssim&{}
\||v|^{2+\gamma-\varepsilon}
\mathbb{1}_{\{|v|\le2R\}}\|_{L^{k'}}
\|f^n\|_{L^\infty_tL^1_{x,v}}
\|f^n\|_{L^1_tL^1_xL^k_v(B^R_{x,v})},
\end{aligned}
\end{equation*}
where $k=\frac{d}{d-2}$ and $k'=\frac{d}{2}$ are defined as in Lemma \ref{lem:sobo}.
By Lemma \ref{lem:sobo} and \eqref{D-n}, we have
\begin{align}
\label{Fisher:uni}
\|f^n\|_{L^1_tL^1_xL^k_v(B^R_{x,v})}
\le C_R\Big(T+\int_0^T\cD_n(f^n)\dd t\Big)
\le C_R,
\end{align}
uniformly in $n$.
Moreover,
$|v|^{2+\gamma-\varepsilon}\mathbb{1}_{\{|v|\le2R\}}\in L^{k'}$,
since
$k'(2+\gamma-\varepsilon)>-d$.
Thus, \eqref{singu} is bounded by $C_R\delta^{\varepsilon/2}$ in both cases.

To justify \eqref{Fisher:uni}, repeat the proof of \cite[Lemma 2.11]{DH25} with radial cutoffs in the relative velocity. Since $A_n=A$ on $\{2n^{-1}\le|v-v_*|\le n\}$, this gives, for every $R>0$ and all sufficiently large $n$,
\begin{align*}
    \int_0^T\int_{B_x^R\times B_v^R}
    \langle v\rangle^\gamma
    |\nabla_v\sqrt{f^n}|^2\dd x\dd v\dd t
    \lesssim_{R,T}
    1+\int_0^T\cD_n(f^n_t)\dd t,
\end{align*}
where the implicit constant is independent of $n$.

Applying this splitting in \eqref{continu-time}, the mass and second-moment bounds control the part away from the diagonal, while \eqref{singu} controls the near-diagonal part
\begin{align*}
\Big|\int_s^t\int_{\Do}Q^n(f_r^n,f_r^n)\varphi\dd x\dd v\dd r\Big|
\le C_\delta|t-s|+C_R\delta^{\varepsilon/2}.
\end{align*}
Choosing $\delta$ small and then $|t-s|$ small ensures the uniform time equicontinuity. The refined Ascoli--Arzelà theorem then implies the same pointwise-in-time and product convergences as in Case \ref{case-1}.

Using the second-moment tail estimate from Case \ref{case-1}, we obtain, for any $\varepsilon>0$,
\begin{align*}
    A_nf^nf^n_*\mathbb{1}_{\{|v-v_*|\ge\varepsilon\}}
    \rightharpoonup
    Aff_*\mathbb{1}_{\{|v-v_*|\ge\varepsilon\}}
\end{align*}
in $L^1([0,T]\times\G)$. Since $A_n$ depending only on $|v-v_*|$, we have 
\begin{align*}
&\Pi_{(v-v_*)^\perp}
    \big(\sqrt{A_n}(\nabla_v-\nabla_{v_*})\sqrt{f^nf^n_*}\big)\\
    =&{}
    \Pi_{(v-v_*)^\perp}
    \big((\nabla_v-\nabla_{v_*})\sqrt{A_nf^nf^n_*}\big).
\end{align*}
As a consequence, we have 
\begin{align*}
\cD_n(f^n)
=2\int_{\G}\kappa
\big|\Pi_{(v-v_*)^\perp}
\big((\nabla_v-\nabla_{v_*})\sqrt{A_nf^nf^n_*}\big)\big|^2\dd\eta.
\end{align*}
Then weak lower semicontinuity on the truncated domain ensures
\begin{align*}
&2\int_0^T\int_{\G\cap\{|v-v_*|>\varepsilon\}}\kappa
\big|\Pi_{(v-v_*)^\perp}
\big((\nabla_v-\nabla_{v_*})\sqrt{Aff_*}\big)\big|^2\dd\eta\dd t\\
&\le
\liminf_{n\to\infty}\int_0^T\cD_n(f^n)\dd t\le C.
\end{align*}
Then Let $\varepsilon\to0$. Lower semicontinuity also transfers the moment and local Sobolev bounds to $f$.

For fixed $\delta$, the coefficients $A_n(1-\psi(v-v_*))\Phi$ are uniformly bounded on bounded velocity sets. Product convergence and the tail estimate from Case \ref{case-1} therefore give convergence away from the diagonal. The limiting near-diagonal entropy pairing is controlled as in \eqref{singu}. Letting first $n\to\infty$ and then $\delta\to0$ ensures
\begin{align*}
&\lim_{n\to\infty}
\int_0^T\int_{\G}\kappa A_nf^nf^n_*\Phi=
-\int_0^T\int_{\G}\kappa ff_*
\widetilde\nabla\log f\cdot\widetilde\nabla\varphi.
\end{align*}
Then the limit satisfies the $\cH$-formulation.

In the homogeneous case, $|\nabla_v\varphi(v)-\nabla_{v_*}\varphi(v_*)|\lesssim |v-v_*|$ reduces the kernel singularity.  See Remark \ref{rmk:homo} for more details.

\end{enumerate}
  
  \medskip

Passing to the limit in \eqref{app:consv-law} and \eqref{app:consv-law-2}, we obtain, for all $t\in[0,T]$,
\begin{equation*}
    \begin{aligned}
&\int_{\Do}|v|^2f_t\dd x\dd v
\le\int_{\Do}|v|^2f_0\dd x\dd v,\\
&\int_{\Do}|x|^2f_t\dd x\dd v
\le C(T)\int_{\Do}(|x|^2+|v|^2)f_0\dd x\dd v.
    \end{aligned}
\end{equation*}
For $\gamma>0$, the uniform bounds \eqref{app-hard} and lower semicontinuity ensure that $f\in L^\infty\big([0,T];L^1_{2,2+\delta}(\Do)\big)$. Under \eqref{k-1} or \eqref{k-3}, the lower bound on $\kappa$ further gives $f\in L^1\big([0,T];L^1_{2,2+\gamma+\delta}(\Do)\big)$.
Conservation of mass and momentum for the approximate solutions and the uniform second-moment bound ensure
\begin{equation*}
\int_{\Do}(1,v)f_t\dd x\dd v
=
\int_{\Do}(1,v)f_0\dd x\dd v
\quad\forall\,t\in[0,T].
\end{equation*}

For the entropy inequality \eqref{app:consv-law-3}, the construction of $f_0^n$ and the initial entropy and moment bounds give $\cH(f^n_0)\to\cH(f_0)$. The pointwise-in-time weak convergence established above allows us to apply weak lower semicontinuity to $\cH$.
For the dissipation, for arbitrary fixed $\varepsilon\in(0,1)$, the weak compactness of $f^nf^n_*$ ensures that 
\begin{align*}
 A_nf^nf^n_*
 \mathbb{1}_{\{\varepsilon<|v-v_*|<\varepsilon^{-1}\}}
 \rightharpoonup
 Aff_*
 \mathbb{1}_{\{\varepsilon<|v-v_*|<\varepsilon^{-1}\}}
\end{align*}
in $L^1([0,T]\times\G)$.
The square-root gradient representation used in Case \ref{case-2} and weak lower semicontinuity therefore ensure, as $\varepsilon\to0$,
\begin{align*}
\int_0^T\cD(f_t)\dd t
\le
\liminf_{n\to\infty}\int_0^T\cD_n(f^n_t)\dd t.
\end{align*}
Passing to the limit in \eqref{app:consv-law-3} to derive
\begin{align*}
    \cH(f_T)-\cH(f_0)+\int_0^T\cD(f_t)\dd t\le0.
\end{align*}
The inequality holds with $T$ replaced by any $t\in[0,T]$. In particular, we have  $\cD(f)\in L^1([0,T])$.

It remains to identify the weak formulation with \eqref{FL:H:ab} in Case \ref{case-1}. We show
\begin{equation}
    \label{proof-id}
\begin{aligned}
&\int_0^T\int_{\G}\kappa ff_*
\widetilde\nabla\log f\cdot\widetilde\nabla\varphi\dd\eta\dd t\\
=&{}2\int_0^T\int_{\G}\kappa\sqrt{Aff_*}\,
\Pi_{(v-v_*)^\perp}
(\nabla_v-\nabla_{v_*})\sqrt{ff_*}
\cdot\widetilde\nabla\varphi\dd\eta\dd t\\
=&{}-\int_0^T\int_{\G}\kappa Aff_*
(\nabla_v-\nabla_{v_*})\cdot
\big(\Pi_{(v-v_*)^\perp}
(\nabla_v\varphi-\nabla_{v_*}\varphi_*)\big)
\dd\eta\dd t.
\end{aligned}
\end{equation}
The first equality follows from the local Fisher-information bounds and the weak chain rule for $\sqrt{ff_*}$. For the second, we insert cutoffs away from $v=v_*$ and at infinity, then integrate by parts in the velocity variables.

By using Cauchy--Schwarz inequality, we have 
\begin{align*}
&\Big|\int_0^T\int_{\G}\kappa ff_*
\widetilde\nabla\log f\cdot\widetilde\nabla\varphi
\dd\eta\dd t\Big|\\
\le&{}
\Big(2\int_0^T\cD(f_t)\dd t\Big)^{\frac12}
\Big(\int_0^T\int_{\G}\kappa ff_*
|\widetilde\nabla\varphi|^2\dd\eta\dd t\Big)^{\frac12}.
\end{align*}
We note that the energy bounds and Claim \ref{sol:moments} ensure that 
\begin{align*}
\int_0^T\int_{\G}\kappa ff_*|\widetilde\nabla\varphi|^2
\dd\eta\dd t
\lesssim_{\varphi}
\int_0^T\int_{\G}\langle v\rangle^{2+\gamma_+}ff_*
<\infty.
\end{align*}
The last term in \eqref{proof-id} is integrable by the same bounds. Local boundedness of $A\Phi$ allows removal of the diagonal cutoff, while the energy bound and Claim \ref{sol:moments} control the cutoff at infinity. Thus, \eqref{proof-id} holds, completing the proof.

\end{proof}

We are left to show Claim \ref{sol:moments}.
\begin{proof}[Proof of Claim \ref{sol:moments}]

We first show
\eqref{app:consv-law:max} holds in the Maxwellian potential case. 
The construction of $\Psi$, followed by normalisation, can be found in
\cite[Appendix]{MW99}, applied to the velocity marginal of $f_0$.
By construction of $f^n_0$, we have 
$$
\sup_n\int_{\Do}\Psi(|v|^2)f^n_0\dd x\dd v<\infty.
$$
We write $\Psi=\Psi(|v|^2)$. Testing the approximation equation by
$\Psi$, 
we obtain
\begin{equation*}
\begin{aligned}
&\int_{\Do}(f^n_t-f^n_0)\Psi\dd x\dd v\\
&=
\frac1n\int_0^t\int_{\Do}f^n_s\Delta_v\Psi\dd x\dd v\dd s\\
&\quad+
\frac12\int_0^t\int_{\G}\kappa A_n f^nf^n_*
(\nabla_v-\nabla_{v_*})\cdot
\big(\Pi_{(v-v_*)^\perp}
(\nabla_v\Psi-\nabla_{v_*}\Psi_*)\big)
\dd\eta\dd s.
\end{aligned}
\end{equation*}
Since $\Psi'(0)=0$ and $0\le\Psi''\le1$, we have
$$
0\le\Delta_v\Psi
=2d\Psi'(|v|^2)+4|v|^2\Psi''(|v|^2)
\le(2d+4)|v|^2.
$$
Hence, the diffusion contribution is bounded by
$$
\frac{C}{n}\int_0^t\int_{\Do}
\langle v\rangle^2f^n_s\dd x\dd v\dd s.
$$

We note that $\nabla_v\Psi=2\Psi'(|v|^2)v$.
By the same computation as in \eqref{M-0} below, with
$s\langle v\rangle^{s-2}$ replaced by $2\Psi'(|v|^2)$,
the monotonicity of $\Psi'$ implies
$$
\big(\Psi'(|v|^2)-\Psi'(|v_*|^2)\big)
\big(|v|^2-|v_*|^2\big)\ge0.
$$
Thus, the corresponding contribution is nonpositive after
symmetrization. Using $0\le A_n\le A$ and $\Psi''\ge0$, we obtain
\begin{equation*}
\begin{aligned}
&\int_{\Do}f^n_t\Psi\dd x\dd v\\
&\le
\int_{\Do}f^n_0\Psi\dd x\dd v
+\frac{C}{n}\int_0^t
\|\langle v\rangle^2f^n_s\|_{L^1_{x,v}}\dd s\\
&\quad+
4\int_0^t\int_{\G}\kappa f^nf^n_*
\frac{A(|v-v_*|)}{|v-v_*|^2}\Psi''(|v|^2)
\big(|v|^2|v_*|^2-(v\cdot v_*)^2\big)
\dd\eta\dd s\\
&\le
\int_{\Do}f^n_0\Psi\dd x\dd v
+C(T)\big(1+
\|\langle v\rangle^2f^n\|_{L^\infty_tL^1_{x,v}}\big)^2.
\end{aligned}
\end{equation*}
Indeed, the kernel assumptions with $\gamma=0$ ensure
$A(|v-v_*|)/|v-v_*|^2\lesssim 1$, and we have
$|v|^2|v_*|^2-(v\cdot v_*)^2
\lesssim \langle v\rangle^2\langle v_*\rangle^2.
$
Combining this estimate with the uniform mass and energy bounds
and the uniform initial $\Psi$-moment bound, we conclude
\eqref{app:consv-law:max}.

In the hard potential case $\gamma>0$, the same argument as in
Lemma \ref{lem:M-1}-$(3)$ and \eqref{bounds:total:1} below implies that 
\begin{align*}
    \sup_n\Big(
    \sup_{0\le t\le T}M_{2+\delta}(f^n_t)
    +\int_0^T\int_{\G}
    \kappa\chi_n(|v-v_*|)f^nf^n_*
    \langle v\rangle^{2+\gamma+\delta}
    \dd\eta\dd t
    \Big)\le C_T.
\end{align*}
We then pass to the limit as $n\to\infty$ by using the weak lower
semicontinuity of the moment functionals and the localised
weak convergence of the products $f^nf^n_*$.
\end{proof}
%\end{proof}

\section{Propagation of regularities}
\label{sec:reg}
In this Section, we show apriori estimates of propagation of moments and $L^p$ bounds in Section \ref{sec:moment} and Section \ref{sec:Lp}, respectively. In this subsection, we specify the kernels in each statement.
We consider the interaction kernels $A$ satisfying \eqref{A-1}
or \eqref{A-3} that
\begin{align*}
A(|z|)=|z|^{2+\gamma}
\quad\text{or}\quad
A(|z|)\sim|z|^2\langle z\rangle^\gamma,
\end{align*}
and the spatial kernels $\kappa$ satisfying \eqref{k-1} or \eqref{k-3} that 
\begin{align*}
   \kappa(x)\sim1\quad \text{on $\R^d$ or on $\T^d$}.
\end{align*}

In this section, we establish a priori estimates for the propagation
of velocity moments and $L^p$-norms.
To justify these estimates rigorously, one can apply the apriori estimates first for the smooth approximate solutions constructed
in Section \ref{sec:sol}, including the artificial diffusion
term.
Using the uniform coercivity, Fisher-information and
entropy-dissipation bounds established in Section \ref{sec:existence},
we obtain estimates independent of the approximation parameter.
We then pass to the limit using weak convergence and the weak
lower semicontinuity of the corresponding functionals.
For simplicity, we present the estimates for $f$ directly.

\subsection{Moment estimates}\label{sec:moment}

In the following, we study the hard and moderately soft potential cases in Lemma \ref{lem:M-1}; the very soft potential cases are discussed in Lemma \ref{lem:M-2}. In this section, we show the propagation of the moments
\begin{align*}
    M_s(f)=\int_{\Do}\langle v\rangle^s f\dd x\dd v.
\end{align*}
In the homogeneous case, the propagation of moments has been studied,
in particular in dimension $d=3$, in, for example,
\cite{DV00a,Wu14,Des15,TV00}. We follow the arguments for the homogeneous
case with appropriate modifications for the spatial variable $x$ and
all dimensions $d\ge2$.

We recall the fuzzy Landau equation \eqref{Landau-fuz}
\begin{equation*}
    (\d_t+v\cdot\nabla_x)f
    =\div_v\Big(\int_{\Do}\kappa
    (af_*\nabla_v f-bf_*f)\dd x_*\dd v_*\Big),
\end{equation*}
where
$a_{ij}(v-v_*)=\big(\delta_{ij}
-\frac{(v-v_*)_i(v-v_*)_j}{|v-v_*|^2}\big)A(|v-v_*|)$
and
$b_i(v-v_*)=-(d-1)\frac{(v-v_*)_i}{|v-v_*|^2}A(|v-v_*|)$.
To derive apriori estimates for $M_s(f)$, we formally test the fuzzy
Landau equation with $\langle v\rangle^s$ to obtain
\begin{equation}
 \label{ap:Ms}
\begin{aligned}
    \frac{d}{dt}M_s(f)
    =&\frac12\sum_{i,j=1}^d\int_{\G}\kappa ff_* a_{ij}
    (\d_{ij}\langle v\rangle^s
    +(\d_{ij}\langle v\rangle^s)_*)\dd\eta\\
    &+\sum_{i=1}^d\int_{\G}\kappa ff_* b_i
    (\d_i\langle v\rangle^s
    -(\d_i\langle v\rangle^s)_*)\dd\eta.
\end{aligned}
\end{equation}
Notice that
\begin{equation}
\label{ap:s}
    \d_i\langle v\rangle^s=sv_i\langle v\rangle^{s-2}
    \quad\text{and}\quad
    \d_{ij}\langle v\rangle^s
    =\delta_{ij}s\langle v\rangle^{s-2}
    +s(s-2)v_iv_j\langle v\rangle^{s-4}.
\end{equation}

In the following, we study the hard and moderately soft potential cases
in Lemma \ref{lem:M-1}. The very soft potential cases are discussed in
Lemma \ref{lem:M-2}.
\begin{lemma}\label{lem:M-1}
Let $d\ge 2$ and $s>2$. Let $A$ satisfy Assumption \ref{ASS:kernel}. Let $\kappa$ be given as in \eqref{k-1} or \eqref{k-3}. Let $f$ be given as in Theorem \ref{thm:exist}. 
\begin{enumerate}[(1)]
\item Let $A(|v-v_*|)\sim |v-v_*|^{2}\langle v-v_*\rangle^\gamma$.
Let $M_s(f_0)<+\infty$. 
In the case of $\gamma\in(-2,0]$, we have
\begin{align*}
    M_s(f_t)\lesssim1+t.
\end{align*}
In the case of $\gamma\in(-\infty,-2]$, we have
\begin{align*}
    M_s(f_t)\lesssim
    \langle t\rangle^{\max\{1,\frac{s-2}{3}\}}.
\end{align*}
In particular, for $\gamma=0$ and $\gamma=-2$, the corresponding
estimates also hold for the kernel
$A(|v-v_*|)=|v-v_*|^{2+\gamma}$.

\item  Let $A=|v-v_*|^{2+\gamma}$. Let $\gamma\in(-2,0)$. If $M_s(f_0)<+\infty$, then we have 
\begin{align*}
  M_s(f_t)\lesssim
\langle t\rangle^{\max\{1,\frac{s-2}{2-\gamma}\}}.
\end{align*}

     \item  Let $A(|v-v_*|)\sim |v-v_*|^2\langle v-v_*\rangle^{\gamma}$ or $A(|v-v_*|)=|v-v_*|^{2+\gamma}$.  Let $\gamma\in(0,1]$. 
Let $M_s(f_0)<+\infty$. 
Then  we have
\begin{align*}
M_s(f)+ \int_0^t M_{s+\gamma}(f)\lesssim M_s(f_0)\exp(Ct).
\end{align*}

\end{enumerate}
\end{lemma}
\begin{proof}

We follow \cite{DV00a} and test the fuzzy Landau equation \eqref{FL:ab}
with $\langle v\rangle^s$, $s>2$. Without loss of generality, we assume $\|f\|_{L^1(\Do)}=1$ and  $\|\langle v\rangle^2 f\|_{L^1_{x,v}}\le E_0$ for some $E_0>0$. 
Then we have
\begin{align*}
&\frac{d}{dt}M_s(f)=2s\sum_{i=1}^d\int_{\G}\kappa ff_* b_i\langle v\rangle^{s-2}v_i\dd\eta\\
&\quad+s\sum_{i=1}^d\int_{\G}\kappa ff_* a_{ii}\langle v\rangle^{s-2}\dd\eta
+s(s-2)\sum_{i,j=1}^d\int_{\G}\kappa ff_* a_{ij}\langle v\rangle^{s-4}v_iv_j\dd\eta\\
&=-2s(d-1)\sum_{i=1}^d\int_{\G}\kappa ff_*
\frac{A(|v-v_*|)}{|v-v_*|^2}\langle v\rangle^{s-2}(v-v_*)_iv_i\dd\eta\\
&\quad+s(d-1)\int_{\G}\kappa ff_*
\frac{A(|v-v_*|)}{|v-v_*|^2}\langle v\rangle^{s-2}|v-v_*|^2\dd\eta\\
&\quad+s(s-2)\sum_{i,j=1}^d\int_{\G}\kappa ff_*
\frac{A(|v-v_*|)}{|v-v_*|^2}\langle v\rangle^{s-4}v_iv_j
\big(|v-v_*|^2\delta_{ij}-(v-v_*)_i(v-v_*)_j\big)\dd\eta.
\end{align*}
Notice that
\begin{gather*}
|v-v_*|^2-2\sum_{i=1}^d(v-v_*)_iv_i=|v_*|^2-|v|^2,\\
\sum_{i,j=1}^d
\big(|v-v_*|^2\delta_{ij}-(v-v_*)_i(v-v_*)_j\big)v_iv_j
=|v|^2|v_*|^2-(v\cdot v_*)^2.
\end{gather*}
Hence, we have
\begin{equation}
\label{M-0}
\begin{aligned}
\frac{d}{dt}M_s(f)
=&{}s\int_{\G}\kappa ff_*
\frac{A(|v-v_*|)}{|v-v_*|^2}\langle v\rangle^{s-2}
\Big((d-1)(|v_*|^2-|v|^2)\\
&\quad+(s-2)\frac{|v|^2|v_*|^2-(v\cdot v_*)^2}{1+|v|^2}\Big)\dd\eta.
\end{aligned}
\end{equation}
Since $s>2$ and
$\frac{|v|^2|v_*|^2-(v\cdot v_*)^2}{1+|v|^2}\le|v_*|^2$,
we have
\begin{equation}
\label{Ms}
\begin{aligned}
\frac{d}{dt}M_s(f)
\leq{}s\int_{\G}\kappa ff_*
\frac{A(|v-v_*|)}{|v-v_*|^2}\langle v\rangle^{s-2}
\big(-(d-1)|v|^2+(s+d-3)|v_*|^2\big)\dd\eta.
\end{aligned}
\end{equation}
To proceed, we consider different cases.
\begin{enumerate}[(1)]

\item

In the case of $\gamma\in(-2,0]$, the estimate \eqref{Ms},
together with the kernel bounds and the uniform mass and energy bounds,
implies that
\begin{equation}
\label{Ms-0}
    \frac{d}{dt}M_s(f)
    \le -C_1M_{s+\gamma}(f)+C_2M_{s-2}(f)
\end{equation}
for some constants $C_1,\,C_2>0$, where we use Peetre's inequality
\begin{equation*}
\label{petree}
    2^{-|\gamma|/2}\langle v\rangle^\gamma
    \langle v_*\rangle^\gamma
    \le\langle v-v_*\rangle^\gamma.
\end{equation*}
By Hölder's inequality with exponents
$\frac{s+\gamma}{2+\gamma}$ and $\frac{s+\gamma}{s-2}$,
we have
\begin{equation*}
\label{hy}
    M_{s-2}(f)
    \le M_0(f)^{\frac{2+\gamma}{s+\gamma}}
    M_{s+\gamma}(f)^{\frac{s-2}{s+\gamma}}
    =M_{s+\gamma}(f)^{\frac{s-2}{s+\gamma}},
\end{equation*}
where we use $M_0(f)=1$.
Then, by Young's inequality, we have
\begin{equation}
\label{young}
    M_{s-2}(f)\le C+ C_1/C_2M_{s+\gamma}(f).
\end{equation}
We substitute \eqref{young} into \eqref{Ms-0} to derive
\begin{align*}
    M_s(f_t)\le M_s(f_0)+Ct.
\end{align*}

In the case of $\gamma\in(-\infty,-2]$, we follow \cite{TV00}
with appropriate modifications to the fuzzy case.
We recall the identity \eqref{M-0}:
\begin{equation}
\label{k:upper:bdd:moment:2}
\begin{aligned}
    &\frac{d}{dt}M_s(f_t)
    =s\int_{\G}\kappa ff_*
    \frac{A(|v-v_*|)}{|v-v_*|^2}\langle v\rangle^{s-2}\\
    &\quad\times
    \Big(-(d-1)\langle v\rangle^2
    +(d-1)\langle v_*\rangle^2
    +(s-2)\frac{|v|^2|v_*|^2-(v\cdot v_*)^2}
    {1+|v|^2}\Big)\dd\eta.
\end{aligned}
\end{equation}
By Hölder's inequality and symmetry, we have
\begin{equation}
\label{h-2}
\begin{aligned}
    &\int_{\G}\kappa ff_*
    \frac{A(|v-v_*|)}{|v-v_*|^2}
    \langle v\rangle^{s-2}\langle v_*\rangle^2\dd\eta\\
    \le&{}
    \Big(\int_{\G}\kappa ff_*
    \frac{A(|v-v_*|)}{|v-v_*|^2}
    \langle v\rangle^s\dd\eta\Big)^{\frac{s-2}{s}}
    \Big(\int_{\G}\kappa ff_*
    \frac{A(|v-v_*|)}{|v-v_*|^2}
    \langle v_*\rangle^s\dd\eta\Big)^{\frac{2}{s}}\\
    =&{}
    \int_{\G}\kappa ff_*
    \frac{A(|v-v_*|)}{|v-v_*|^2}
    \langle v\rangle^s\dd\eta.
\end{aligned}
\end{equation}
We substitute \eqref{h-2} into \eqref{k:upper:bdd:moment:2}
to derive
\begin{equation*}
\label{M-3}
\begin{aligned}
    \frac{d}{dt}M_s(f_t)
    \lesssim&{}
    \int_{\G}\kappa ff_*
    \frac{A(|v-v_*|)}{|v-v_*|^2}
    \langle v\rangle^{s-4}
    \big(|v|^2|v_*|^2-(v\cdot v_*)^2\big)\dd\eta.
\end{aligned}
\end{equation*}

Notice that
\begin{equation*}
\label{b-2}
\begin{aligned}
    &|v|^2|v_*|^2-(v\cdot v_*)^2
    =|v|^2|v_*|^2\sin^2\alpha\\
    =&{}|v||v_*|(|v|\sin\alpha)(|v_*|\sin\alpha)\\
    \le&{}|v||v_*||v-v_*|^2,
\end{aligned}
\end{equation*}
where $\alpha=\arccos\frac{v\cdot v_*}{|v||v_*|}$
for nonzero $v$ and $v_*$, and we use
\begin{align*}
    &|v||v_*|\sin^2\alpha-|v-v_*|^2\\
    =&{}|v||v_*|(\sin^2\alpha+2\cos\alpha)
    -(|v|^2+|v_*|^2)\\
    \le&{}2|v||v_*|-(|v|^2+|v_*|^2)\le0.
\end{align*}
The inequality is immediate if either velocity vanishes.

We use
$A(|v-v_*|)\sim|v-v_*|^2\langle v-v_*\rangle^\gamma\lesssim1$
when $\gamma\in(-\infty,-2]$, and
$A(|v-v_*|)=|v-v_*|^{2+\gamma}=1$ when $\gamma=-2$,
to derive
\begin{equation}
\label{Ms-3}
    \frac{d}{dt}M_s(f_t)
    \lesssim
    \int_{\G}ff_*A(|v-v_*|)
    \langle v\rangle^{s-4}|v||v_*|\dd\eta
    \lesssim M_{s-3}(f),
\end{equation}
where we use the uniform mass and energy bounds.
When $s\in(2,5]$, the right-hand side is bounded by the mass
and energy bounds, and we obtain
\begin{align*}
    M_s(f_t)\le M_s(f_0)+Ct.
\end{align*}
In the case of $s>5$, we write
\begin{align*}
    \langle v\rangle^{s-3}f
    =\langle v\rangle^{s-3-\frac{6}{s-2}}
    f^{\frac{s-5}{s-2}}
    \langle v\rangle^{\frac{6}{s-2}}
    f^{\frac{3}{s-2}}.
\end{align*}
By Hölder's inequality with exponents
$\frac{s-2}{s-5}$ and $\frac{s-2}{3}$, we have
\begin{equation*}
\label{b-3}
    M_{s-3}(f)
    \le M_s(f)^{\frac{s-5}{s-2}}E_0^{\frac{3}{s-2}},
\end{equation*}
which implies that
\begin{align*}
    M_s(f_t)\lesssim\langle t\rangle^{\frac{s-2}{3}}.
\end{align*}

\smallskip

\item
Let $\gamma\in(-2,0)$ and $A(|v-v_*|)=|v-v_*|^{2+\gamma}$.
Using the cancellation in \eqref{M-0}, as in the derivation of
\eqref{Ms-3}, we obtain
\begin{equation*}
\begin{aligned}
    \frac{d}{dt}M_s(f_t)
    \lesssim&{}
    \int_{\G}\kappa ff_*|v-v_*|^\gamma
    \langle v\rangle^{s-4}
    \big(|v|^2|v_*|^2-(v\cdot v_*)^2\big)\dd\eta.
\end{aligned}
\end{equation*}
Notice that, for $v\ne v_*$,
\begin{align*}
    &|v-v_*|^\gamma
    \big(|v|^2|v_*|^2-(v\cdot v_*)^2\big)\\
    =&{}|v-v_*|^{2+\gamma}
    |\Pi_{(v-v_*)^\perp}v|^2\\
    \le&{}(|v|+|v_*|)^{2+\gamma}
    \min(|v|^2,|v_*|^2)\\
    \lesssim&{}|v|^{2+\gamma}|v_*|^2,
\end{align*}
where the last inequality follows from $0<2+\gamma<2$.
Using the upper bound on $\kappa$ and the energy bound, we obtain
\begin{align*}
    \frac{d}{dt}M_s(f_t)
    &\lesssim
    \int_{\G}ff_*\langle v\rangle^{s+\gamma-2}
    \langle v_*\rangle^2\dd\eta\lesssim M_{s+\gamma-2}(f).
\end{align*}
When $s\in(2,4-\gamma]$, the right-hand side is bounded by the
mass and energy bounds, and we obtain
\begin{align*}
    M_s(f_t)\le M_s(f_0)+Ct.
\end{align*}
In the case of $s>4-\gamma$, we write
\begin{align*}
    \langle v\rangle^{s+\gamma-2}f
    =\big(\langle v\rangle^sf\big)^{\frac{s+\gamma-4}{s-2}}
    \big(\langle v\rangle^2f\big)^{\frac{2-\gamma}{s-2}}.
\end{align*}
By Hölder's inequality, we have
\begin{equation*}
    M_{s+\gamma-2}(f)
    \le M_s(f)^{\frac{s+\gamma-4}{s-2}}
    E_0^{\frac{2-\gamma}{s-2}},
\end{equation*}
which implies that
\begin{align*}
    M_s(f_t)\lesssim
    \langle t\rangle^{\frac{s-2}{2-\gamma}}.
\end{align*}
Combining the two cases, we obtain, for all $s>2$,
\begin{align*}
    M_s(f_t)\lesssim
    \langle t\rangle^{\max\{1,\frac{s-2}{2-\gamma}\}}.
\end{align*}
 
\smallskip

\item
Let $\gamma\in(0,1]$.
We follow \cite{DV00a} with appropriate modifications to the fuzzy case.
We only show the case
$\frac{A(|v-v_*|)}{|v-v_*|^2}\sim\langle v-v_*\rangle^\gamma$;
the case $A(|v-v_*|)=|v-v_*|^{2+\gamma}$ follows analogously.
By \eqref{Ms} and the kernel bounds, we have
\begin{align*}
    &\frac{d}{dt}M_s(f_t)
    \lesssim
    \int_{\G}\kappa ff_*\langle v-v_*\rangle^\gamma
    \langle v\rangle^{s-2}
    \big(-|v|^2+C|v_*|^2\big)\dd\eta.
\end{align*}
Exchanging $(x,v)$ and $(x_*,v_*)$ and averaging, we obtain
\begin{align*}
    &\frac{d}{dt}M_s(f_t)
    \lesssim
    \int_{\G}\kappa ff_*\langle v-v_*\rangle^\gamma\\
    &\quad\times
    \big(-\langle v\rangle^s-\langle v_*\rangle^s
    +C\langle v\rangle^{s-2}\langle v_*\rangle^2
    +C\langle v\rangle^2\langle v_*\rangle^{s-2}\big)\dd\eta.
\end{align*}
By a Povzner-type inequality, we have
\begin{align*}
    &-\langle v\rangle^s
    +C\langle v\rangle^{s-2}\langle v_*\rangle^2
    -\langle v_*\rangle^s
    +C\langle v_*\rangle^{s-2}\langle v\rangle^2\\
    \le&{}-K_1\langle v\rangle^s
    +C_1\big(\langle v\rangle^{s-1}\langle v_*\rangle
    +\langle v_*\rangle^{s-1}\langle v\rangle\big)
\end{align*}
for some constants $K_1,\,C_1>0$.
After adjusting the positive constants, we obtain
\begin{equation*}
\label{bounds:1:R3:3}
\begin{aligned}
    \frac{d}{dt}M_s(f_t)
    \le&{}\int_{\G}\kappa ff_*\langle v-v_*\rangle^\gamma\\
    &\quad\times
    \Big(-K_1\langle v\rangle^s
    +C_1\big(\langle v\rangle^{s-1}\langle v_*\rangle
    +\langle v_*\rangle^{s-1}\langle v\rangle\big)\Big)\dd\eta.
\end{aligned}
\end{equation*}
We note that
$\langle v-v_*\rangle^\gamma
\ge\langle v\rangle^\gamma-\langle v_*\rangle^\gamma$.
Hence, using the upper bound on $\kappa$ and symmetry, we have
\begin{equation}
\label{bounds:total:1}
\begin{aligned}
    &\frac{d}{dt}M_s(f_t)
    +K_1\int_{\G}\kappa ff_*\langle v\rangle^{s+\gamma}\dd\eta\\
    \lesssim&{}
    \int_{\G}ff_*
    \big(\langle v\rangle^s\langle v_*\rangle^\gamma
    +\langle v\rangle^{s+\gamma-1}\langle v_*\rangle
    +\langle v\rangle^{s-1}\langle v_*\rangle^{\gamma+1}\big)
    \dd\eta\\
    \lesssim&{}
    M_\gamma(f)M_s(f)
    +M_{s+\gamma-1}(f)M_1(f)
    +M_{s-1}(f)M_{\gamma+1}(f)\\
    \lesssim&{}M_s(f),
\end{aligned}
\end{equation}
where we use the mass and energy bounds and $\langle v\rangle^\gamma-\langle v_*\rangle^\gamma
    \le \langle v-v_*\rangle^\gamma$.

Under \eqref{k-1}, or under \eqref{k-3} on $\T^d$,
we additionally have $\kappa\gtrsim1$.
Using $M_0(f)=1$ and integrating \eqref{bounds:total:1} in time,
we obtain
\begin{align*}
    M_s(f_t)+\int_0^tM_{s+\gamma}(f_r)\dd r
    \lesssim M_s(f_0)\exp(Ct).
\end{align*}

\end{enumerate}
\end{proof}

Based on \eqref{ap:Ms} and \eqref{ap:s}, we have the following
apriori estimates for the very soft potential case
($A=|v-v_*|^{2+\gamma}$).
Notice that these estimates are global in time only for
$\gamma\in(-\min\{d,2\sqrt3\},-2)$.
\begin{lemma}\label{lem:M-2}
Let $A(|v-v_*|)=|v-v_*|^{2+\gamma}$ with $\gamma\in(-\min\{d,2\sqrt3\},-2)$. Let $\kappa$ be given as in \eqref{k-1} and \eqref{k-3}. 
 Let $f$ be given as in Theorem \ref{thm:exist}. Let $s>2$. 
If $M_s(f_0)<+\infty$, then we have $M_s(f)\in L^\infty([0,T])$.

\end{lemma}
\begin{proof}

We follow \cite{Des15} with appropriate modifications to the fuzzy case.
We recall \eqref{ap:Ms} and \eqref{ap:s}.
Notice that
\begin{equation*}
    \frac{|\d_i\langle v\rangle^s
    -(\d_i\langle v\rangle^s)_*|}{|v-v_*|}
    \lesssim_{d,s}
    \langle v\rangle^{s-2}+\langle v_*\rangle^{s-2}.
\end{equation*}
By using \begin{align*}
    |v-v_*||b(v-v_*)|\lesssim A(|v-v_*|),
\end{align*}
together with the upper bound on $\kappa$, we obtain
\begin{equation*}
\label{Ms-1}
    \frac{d}{dt}M_s(f_t)
    \lesssim
    \int_{\G}ff_*|v-v_*|^{2+\gamma}
    \langle v\rangle^{s-2}\dd\eta.
\end{equation*}

We split the domain into $\{|v-v_*|\ge1\}$ and
$\{|v-v_*|<1\}$.
Since $\gamma<-2$ and $M_0(f)=1$, we have
\begin{equation*}
    \int_{\G\cap\{|v-v_*|\ge1\}}
    ff_*|v-v_*|^{2+\gamma}\langle v\rangle^{s-2}\dd\eta
    \le M_{s-2}(f).
\end{equation*}

We choose $\alpha$ such that
$$
\max\left\{0,\frac{|\gamma|^2}{2}-4\right\}<\alpha<2,
$$
and set
$$
\beta=\frac{s+\alpha}{s+|\gamma|},\quad
\frac1{q_1}=\frac{\beta}{k}+1-\beta,\quad
\frac1{q_2}=1-\frac1{q_1}=\frac{2\beta}{d},
$$
where $k=\frac{d}{d-2}$ is given in Lemma \ref{lem:sobo}.
Such a choice of $\alpha$ is possible since $|\gamma|<2\sqrt3$.
Moreover, we have
$$
0<\beta<1,\quad
\alpha=|\gamma|\beta+(\beta-1)s,\quad
\beta>\frac{|2+\gamma|}{2}.
$$
It follows that
$|v|^{2+\gamma}\mathbb{1}_{\{|v|\le1\}}\in L^{q_2}(\R^d)$.

On $\{|v-v_*|<1\}$, we have
$\langle v\rangle^\alpha\langle v_*\rangle^{-\alpha}\le C$.
Hence, by Hölder's inequality and Young's convolution inequality,
we obtain
\begin{equation*}
\begin{aligned}
    &\int_{\G\cap\{|v-v_*|<1\}}
    ff_*|v-v_*|^{2+\gamma}\langle v\rangle^{s-2}\dd\eta\\
    \lesssim&{}
    \int_{\G\cap\{|v-v_*|<1\}}
    ff_*|v-v_*|^{2+\gamma}
    \langle v_*\rangle^{-\alpha}
    \langle v\rangle^{s-2+\alpha}\dd\eta\\
    \lesssim&{}
    \|\langle v\rangle^{-\alpha}f\|_{L^1_xL^{q_1}_v}
    \||v|^{2+\gamma}\mathbb{1}_{\{|v|\le1\}}\|_{L^{q_2}}
    M_{s-2+\alpha}(f).
\end{aligned}
\end{equation*}
By Lemma \ref{lem:sobo}-$(2)$, we have
\begin{align*}
    \|\langle v\rangle^{-\alpha}f\|_{L^1_xL^{q_1}_v}
    \le
    \|\langle v\rangle^\gamma f\|_{L^1_xL^k_v}^{\beta}
    M_s(f)^{1-\beta}.
\end{align*}
Combining these estimates, we obtain
\begin{equation*}
\begin{aligned}
    \frac{d}{dt}M_s(f_t)
    \lesssim{}
    M_{s-2}(f_t)+
    \|\langle v\rangle^\gamma f_t\|_{L^1_xL^k_v}^{\beta}
    M_s(f_t)^{1-\beta}M_{s-2+\alpha}(f_t).
\end{aligned}
\end{equation*}
By the global weighted estimate in Lemma \ref{lem:sobo} and the
entropy inequality, we have $\int_0^T
\|\langle v\rangle^\gamma f_t\|_{L^1_xL^k_v}\dd t\lesssim 1+\int_0^T\cD(f)<+\infty$.

For $2<s\le4-\alpha$, both $M_{s-2}(f)$ and
$M_{s-2+\alpha}(f)$ are bounded by the mass and energy bounds.
Grönwall's inequality implies
$M_s(f)\in L^\infty([0,T])$ in this range.
For larger $s$, we repeat the argument in increments of $2-\alpha$:
the moments $M_{s-2}(f)$ and $M_{s-2+\alpha}(f)$ are bounded
at the preceding step, and the same Grönwall argument applies.
Thus, for every $s>2$ such that $M_s(f_0)<\infty$, we obtain
$M_s(f)\in L^\infty([0,T])$.

\end{proof}

\subsection{\texorpdfstring{$L^p$}{Lp}-estimates}\label{sec:Lp}
In this section, we study the propagation of $L^p$-norms. In the homogeneous case, the propagation of $L^p$-norms has been studied, in particular in dimension $d=3$, in, for example, \cite{DV00a,Wu14,FG09}, for hard and soft potentials. We follow the arguments for homogeneous cases with appropriate modifications for the spatial variable $x$ and all dimensions $d\ge2$.

We derive apriori estimates by multiplying the fuzzy Landau equation \eqref{Landau-fuz} by $f^{p-1}$:
\begin{equation}
\label{p}
    p^{-1}\frac{d}{dt}\|f\|_{L^p}^p
    =-(p-1)\int_{\Do}f^{p-2}(\nabla_v f)^T\overline a\nabla_v f\dd x\dd v
    -\frac{p-1}{p}\int_{\Do}f^p\overline c\dd x\dd v.
\end{equation}
 We discuss the spatial kernels
\begin{equation*}
    \kappa(x)\sim1
    \quad\text{and}\quad
    \kappa(x)\equiv1
\end{equation*}
in Lemma \ref{lem:Lp-2} and Lemma \ref{lem:Lp-3}, respectively.

\begin{lemma}\label{lem:Lp-2}
Let $A$ satisfy Assumption \ref{ASS:kernel}. Let $\kappa$ be given as in \eqref{k-1} and \eqref{k-3}.  Let $d\ge2$. Let $p\in[2,+\infty)$.
Let $f$ be given as in Theorem \ref{thm:exist}.
\begin{enumerate}[(1)]
    \item
    Let $\gamma\in(-2,0]$. If $f_0\in L^p(\Do)$, then we have
    $f\in L^\infty([0,T];L^p(\Do))$ and
    $\langle v\rangle^{\frac{\gamma}{2}}\nabla_vf^{\frac p2}
    \in L^2([0,T]\times\Do)$.

    \item
    Let $d\ge3$. Let $\gamma=-2$. Let $\delta>0$.
    If $f_0\in L^p(\Do)$ and $M_{2+\delta}(f_0)<+\infty$,
    then we have
    $f\in L^\infty([0,T];L^p(\Do))$ and
    $\langle v\rangle^{-1}\nabla_vf^{\frac p2}
    \in L^2([0,T]\times\Do)$.
\end{enumerate}
\end{lemma}

\begin{proof}
We recall \eqref{p}:
\begin{equation*}
\label{A-B-2}
\begin{aligned}
    &p^{-1}\frac{d}{dt}\int_{\Do}f^p\dd x\dd v\\
    =&{}-(p-1)\int_{\Do}
    f^{p-2}(\nabla_vf)^T\overline a\nabla_vf\dd x\dd v
    -\frac{p-1}{p}\int_{\Do}f^p\overline c\dd x\dd v\\
    =:&{}-A+B.
\end{aligned}
\end{equation*}
We follow \cite[Theorem 2]{Wu14} with appropriate modifications
for the spatial variable and the dimension.

\begin{enumerate}[(1)]

\item
By the coercivity Lemma \ref{lem:coer}, we have
\begin{equation}
\label{A:coe-0}
\begin{aligned}
    A
    &\ge(p-1)C_{coe}\int_{\Do}
    \langle v\rangle^\gamma f^{p-2}|\nabla_vf|^2\dd x\dd v\\
    &=\frac{4(p-1)C_{coe}}{p^2}
    \|\langle v\rangle^{\frac{\gamma}{2}}
    \nabla_vf^{\frac p2}\|_{L^2(\Do)}^2.
\end{aligned}
\end{equation}
Since
$$
|\nabla_v(\langle v\rangle^{\frac{\gamma}{2}}f^{\frac p2})|
\lesssim
\langle v\rangle^{\frac{\gamma}{2}}
|\nabla_vf^{\frac p2}|
+\langle v\rangle^{\frac{\gamma}{2}-1}f^{\frac p2},
$$
we obtain
\begin{equation*}
    A\ge
    C^{-1}\|\nabla_v(\langle v\rangle^{\frac{\gamma}{2}}
    f^{\frac p2})\|_{L^2(\Do)}^2
    -C\|f\|_{L^p(\Do)}^p.
\end{equation*}
The upper bound on $\kappa$ and the bound on $\bar{c}$ ensure
$$
|B|\lesssim\int_{\G}f^pf_*|v-v_*|^\gamma\dd\eta.
$$
For $\gamma=0$, the mass conservation implies
$|B|\lesssim\|f\|_{L^p(\Do)}^p$, and the conclusion follows
from Grönwall's inequality.
We therefore consider $\gamma\in(-2,0)$.

We split $\G$ into $\{|v-v_*|\le1\}$ and $\{|v-v_*|>1\}$.
On $\{|v-v_*|>1\}$, the mass conservation implies
\begin{align*}
    \int_{\G\cap\{|v-v_*|>1\}}
    f^pf_*|v-v_*|^\gamma\dd\eta
    \le\|f\|_{L^p(\Do)}^p.
\end{align*}
On $\{|v-v_*|\le1\}$, we have
$\langle v_*\rangle^\gamma\langle v\rangle^{-\gamma}\le C$.
Hence, we have 
\begin{align*}
    &\int_{\G\cap\{|v-v_*|\le1\}}
    f^pf_*|v-v_*|^\gamma\dd\eta\\
    \lesssim&{}
    \|\langle v\rangle^{-\gamma}f\|_{L^1(\Do)}
    \sup_{v_*\in\R^d}
    \int_{\Do\cap\{|v-v_*|\le1\}}
    |v-v_*|^\gamma
    (\langle v\rangle^{\frac{\gamma}{2}}f^{\frac p2})^2
    \dd x\dd v.
\end{align*}
Since $-\gamma<2$, the first factor is bounded by the energy.

We denote by $\hat f$ the Fourier transform in $v$:
\begin{equation*}
    \hat f(x,\xi)
    =(2\pi)^{-\frac d2}\int_{\R^d}
    f(x,v)e^{-iv\cdot\xi}\dd v.
\end{equation*}
By Pitt's inequality, see, for example,
\cite{Bec08,Bec08b,Bec12}, we have, uniformly in $v_*\in\R^d$,
\begin{equation*}
\label{soft-pitt}
\begin{aligned}
\int_{\Do}|v-v_*|^\gamma
    (\langle v\rangle^{\frac{\gamma}{2}}f^{\frac p2})^2
    \dd x\dd v
    \le{}C_{pitt}\int_{\Do}|\xi|^{-\gamma}
    \Big|\widehat{\langle v\rangle^{\frac{\gamma}{2}}
    f^{\frac p2}}\Big|^2\dd x\dd\xi.
\end{aligned}
\end{equation*}
Here translation in $v$ only introduces a Fourier phase
of modulus one.

We split the frequency domain into $\{|\xi|\le R\}$ and
$\{|\xi|>R\}$.
By Parseval's theorem, we have
\begin{align*}
\int_{\Do\cap\{|\xi|\le R\}}|\xi|^{-\gamma}
    \Big|\widehat{\langle v\rangle^{\frac{\gamma}{2}}
    f^{\frac p2}}\Big|^2\dd x\dd\xi
    \le {}R^{|\gamma|}
    \|\langle v\rangle^{\frac{\gamma}{2}}f^{\frac p2}
    \|_{L^2(\Do)}^2
    \le R^{|\gamma|}\|f\|_{L^p(\Do)}^p.
\end{align*}
Similarly, we have
\begin{align*}
\int_{\Do\cap\{|\xi|>R\}}|\xi|^{-\gamma}
    \Big|\widehat{\langle v\rangle^{\frac{\gamma}{2}}
    f^{\frac p2}}\Big|^2\dd x\dd\xi
    \le{}R^{|\gamma|-2}
    \|\nabla_v(\langle v\rangle^{\frac{\gamma}{2}}
    f^{\frac p2})\|_{L^2(\Do)}^2.
\end{align*}
Hence, by choosing $R$ sufficiently large, we obtain
\begin{align*}
    &\frac{d}{dt}\|f\|_{L^p(\Do)}^p
    +C^{-1}
    \|\nabla_v(\langle v\rangle^{\frac{\gamma}{2}}
    f^{\frac p2})\|_{L^2(\Do)}^2
    \lesssim\|f\|_{L^p(\Do)}^p.
\end{align*}
Grönwall's inequality and the relation between the weighted
gradients then ensure
$f\in L^\infty([0,T];L^p(\Do))$ and
$\langle v\rangle^{\frac{\gamma}{2}}\nabla_vf^{\frac p2}
\in L^2([0,T]\times\Do)$.

\smallskip

\item
Let $\gamma=-2$ and $d\ge3$.
We only show the case $x\in\R^d$; the case $x\in\T^d$
follows analogously without the spatial moment or spatial weight.
We follow the preceding argument, treating the near-diagonal
collision term differently.
For $M>0$, we split the domain $\G\cap\big\{|v-v_*|\le1\big\}$ into
\begin{align*}
    &\Omega_1=\G\cap\big\{|v-v_*|\le1\big\}
    \cap\big\{\log f_*+|x_*|^2\le M\big\},\\
    &\Omega_2=\G\cap\big\{|v-v_*|\le1\big\}
    \cap\big\{\log f_*+|x_*|^2>M\big\}.
\end{align*}
On $\Omega_1$, we have $0\le f_*\le\exp(M-|x_*|^2)$ and
\begin{align*}
    \int_{\Omega_1}|v-v_*|^{-2}f^pf_*\dd\eta
    &\lesssim
    e^M\int_{\{|v|\le1\}}|v|^{-2}\dd v
    \|f\|_{L^p_{x,v}}^p\\
    &\lesssim e^M\|f\|_{L^p_{x,v}}^p.
\end{align*}
On $\{|v-v_*|\le1\}$, we have
$\langle v\rangle\sim\langle v_*\rangle$.
Thus, on $\Omega_2$, we obtain
\begin{equation}
\label{omega-2}
\begin{aligned}
&\int_{\Omega_2}|v-v_*|^{-2}f^pf_*\dd\eta
    \lesssim{}
    \int_{\{f>\exp(M-|x|^2)\}}
    \langle v\rangle^2f\dd x\dd v\\
    &\quad\times
    \sup_{v_*\in\R^d}
    \int_{\Do}|v-v_*|^{-2}
    (\langle v\rangle^{-1}f^{\frac p2})^2\dd x\dd v.
\end{aligned}
\end{equation}
Since $d\ge3$, Pitt's inequality and Parseval's theorem ensure
\begin{align*}
    \int_{\Do}|v-v_*|^{-2}
    (\langle v\rangle^{-1}f^{\frac p2})^2\dd x\dd v
    \le&\, C_{pitt}\int_{\Do}|\xi|^2
    \Big|\widehat{\langle v\rangle^{-1}f^{\frac p2}}
    \Big|^2\dd x\dd\xi\\
    =&\, C_{pitt}
    \|\nabla_v(\langle v\rangle^{-1}f^{\frac p2})
    \|_{L^2(\Do)}^2.
\end{align*}
By \eqref{A:coe-0} and the weighted gradient estimate above,
the last expression is bounded by
$C(A+\|f\|_{L^p}^p)$.
Consequently, the contribution of $\Omega_2$ can be absorbed
into $A$ once
$\int_{\{f>\exp(M-|x|^2)\}}\langle v\rangle^2f\dd x\dd v$
is sufficiently small.

The entropy and moment bounds ensure
$$
\sup_{0\le t\le T}
\int_{\Do}f_t\big(\langle x\rangle^2+(\log f_t)^+\big)
\dd x\dd v\le C(T).
$$
Hence, we have
\begin{equation}
\label{-2:1}
    \int_{\{f_t>\exp(M-|x|^2)\}}f_t\dd x\dd v
    \le\frac{C(T)}{M}
    \quad\forall t\in[0,T].
\end{equation}
By Lemma \ref{lem:M-1}, the assumed initial
$(2+\delta)$-moment implies
$M_{2+\delta}(f)\in L^\infty([0,T])$.
Thus, for $R\ge1$, we have
\begin{equation}
\label{-2:2}
    \int_{\R^d_x\times(B^R_v)^c}
    \langle v\rangle^2f_t\dd x\dd v
    \le C(T)R^{-\delta}
    \quad\forall t\in[0,T].
\end{equation}
Combining \eqref{-2:1} and \eqref{-2:2}, we obtain
\begin{align*}
    \int_{\{f_t>\exp(M-|x|^2)\}}
    \langle v\rangle^2f_t\dd x\dd v
    \le C(T)R^{-\delta}+\frac{C(T)R^2}{M}
    \quad\forall t\in[0,T].
\end{align*}
We choose $R$ sufficiently large and then $M$
sufficiently large to derive
\begin{equation*}
\label{e:M}
\begin{aligned}
    &\frac{d}{dt}\|f\|_{L^p_{x,v}}^p
    +C^{-1}
    \|\nabla_v(\langle v\rangle^{-1}f^{\frac p2})
    \|_{L^2(\Do)}^2
    \lesssim e^M\|f\|_{L^p_{x,v}}^p.
\end{aligned}
\end{equation*}
Grönwall's inequality and the relation between the weighted
gradients imply that
\begin{equation*}
\label{exp}
\begin{aligned}
    &\|f_t\|_{L^p_{x,v}}^p
    +C^{-1}\int_0^t
    \|\langle v\rangle^{-1}\nabla_vf_s^{\frac p2}
    \|_{L^2_{x,v}}^2\le
    C\|f_0\|_{L^p_{x,v}}^p\exp(Ce^Mt).
\end{aligned}
\end{equation*}
Here $M$ and the constants may depend on $T$ through the
uniform moment and coercivity bounds.

\end{enumerate}
\end{proof}

We show $L^p$ apriori estimates in the case of $\kappa\equiv 1$ in Lemma \ref{lem:Lp-3}. We closely follow the homogeneous arguments in \cite{DV00a} and \cite{ALL15}, where in particular, $d=3$ has been discussed. For the reason of completeness, we show the detailed proof for all dimensions $d\ge 2$. Let $\kappa\equiv1$ in \eqref{Landau-fuz}, and define the
velocity marginal
$$
    \hf _t(v)=\big(f_t*_x\kappa\big)(v)=\int_x f_t(x,v)\dd x.
$$
Integrating \eqref{Landau-fuz} in $x$, we obtain, at least formally,
the homogeneous Landau equation
\begin{equation}
\label{Landau:homo}
\begin{cases}
    \d_t\hf =Q_{\sf homo}(\hf ,\hf ),\\
    \hf |_{t=0}=\displaystyle\int_x f_0(x,v)\dd x,
\end{cases}
\end{equation}
where
\begin{equation*}
\begin{aligned}
    Q_{\sf homo}(\hf ,\hf )
    =\nabla_v\cdot\Big(
    \int_{\R^d}A(|v-v_*|)\Pi_{(v-v_*)^\perp}
    \big(\hf (v_*)\nabla_v\hf (v)
    -\hf (v)\nabla_{v_*}\hf (v_*)\big)\dd v_*
    \Big).
\end{aligned}
\end{equation*}

When the spatial domain is $\T^d$, $\hf $ is an $\cH$-solution of
\eqref{Landau:homo}, with the homogeneous definition understood
analogously to Definition \ref{def:H}.
To see this, we define
\begin{align*}
    \cH_{homo}(\hf )
    &=\int_{\R^d}\hf \log \hf \dd v,\\
    \cD_{homo}(\hf )
    &=\frac12\int_{\R^{2d}}A(|v-v_*|)
    \frac{
    \big|\Pi_{(v-v_*)^\perp}
    (\nabla_v-\nabla_{v_*})(\hf \hf _*)\big|^2
    }{\hf \hf _*}\dd v_*\dd v.
\end{align*}
By the convexity of $f\mapsto f\log f$, the joint convexity
of $(a,b)\mapsto|a|^2/b$, and Jensen's inequality, we obtain
\begin{equation}
\label{fuz-homo}
    \cH_{homo}(\hf )\le\cH(f),
    \quad
    \cD_{homo}(\hf )\le\cD(f).
\end{equation}
Moreover, choosing test functions $\varphi=\varphi(t,v)$
independent of $x$ in \eqref{FL:H:ab} and integrating the
collision pairing over $(x,x_*)$ implies the homogeneous
$\cH$-formulation.
In the whole space case $\R^d_x$, the dissipation inequality
in \eqref{fuz-homo} remains valid, while the entropy inequality
holds up to an additive constant depending on
$\|\langle x\rangle^2f\|_{L^1_{x,v}}$; see, for example,
\cite[Corollary 2.11]{DH25}.

In Lemma \ref{lem:Lp-3}, we establish $L^p$ apriori estimates
for the velocity marginal $\hf$ when $\kappa\equiv1$.
We follow the homogeneous arguments in \cite{DV00a,ALL15},
which treat in particular the case $d=3$.
For completeness, we provide the proof for all dimensions $d\ge2$.

\begin{lemma}\label{lem:Lp-3}
Let $\kappa\equiv1$.
Let $A$ satisfy Assumption~\ref{ASS:kernel}.
Let $f$ be an $\cH$-solution to the fuzzy Landau equation
\eqref{Landau-fuz}.
\begin{enumerate}[(1)]

\item Let $d\ge2$.
Let $\gamma\in(0,1]$.
Let $p\in[2,+\infty)$.
Let $f_0\in L^p_vL^1_x$.
In addition, when $d\ge5$, we assume
$M_{\gamma(1+\frac{p-1}{p}\frac{d-2}{2})}(f_0)<+\infty$.
Then we have
$f\in L^\infty([0,T];L^p_vL^1_x)$.

\item Let $d\ge 3$.  Let $\gamma\in(-\gamma_d,-2)$. Let $k=\frac d{d-2}$.
Let $p>\frac d2$ and $s\ge2+k(|\gamma|-2)$.
Let $\langle v\rangle^sf_0\in L^p_vL^1_x$ such that $\|\langle v\rangle^sf_0\|_{L^p_vL^1_x}\le\delta$ for some sufficiently small $\delta>0$. Then we have $\langle v\rangle^sf\in L^\infty([0,T];L^p_vL^1_x)$.

\end{enumerate}
\end{lemma}

\begin{proof}
\begin{enumerate}[(1)]
\item
We write $\hf(v)=f*_x\kappa=\int_{\R^d}f(x,v)\dd x$, which satisfies the homogeneous Landau equation
\begin{equation}
    \label{HL:ab-1}
    \d_t\hf=\div_v\big(\overline a\nabla_v \hf-\overline b \hf\big),
\end{equation}
where $\overline a=a*_v\hf$, $\overline b=b*_v\hf$ and
$\overline c=c*_v\hf$.
We multiply \eqref{HL:ab-1} by $\hf^{p-1}$ and integrate over $\R^d$ to derive
\begin{equation}
    \label{A-B}
    \begin{aligned}
       p^{-1}\frac{d}{dt}\int_{\R^d}\hf^p\dd v
       & =-(p-1)\int_{\R^d}\hf^{p-2}
       (\nabla_v \hf)^T\overline a\nabla_v \hf\dd v\\
       &\quad-\frac{p-1}{p}\int_{\R^d}\hf^p\overline c\dd v
       =:-A+B.
    \end{aligned}
\end{equation}

We follow \cite[Theorem 5]{DV00a} with appropriate modifications
for all dimensions $d\ge2$.
We show a lower bound for $A$ and an upper bound for $|B|$.

By the coercivity of $\overline a$ in Lemma \ref{lem:coer}, we have
\begin{equation*}
\label{A:coe-1}
\begin{aligned}
    A&\ge C_{coe}\int_{\R^d}\langle v\rangle^\gamma
    \hf^{p-2}|\nabla_v \hf|^2\dd v\\
    &=\frac{4C_{coe}}{p^2}
    \|\langle v\rangle^{\frac{\gamma}{2}}
    \nabla_v \hf^{\frac p2}\|_{L^2_v}^2.
\end{aligned}
\end{equation*}

Concerning the term $B$, the coefficient bounds ensure
\begin{equation*}
\label{gamma:+:B}
\begin{aligned}
    |B|&\lesssim
    \int_{\R^{2d}}\hf^p\hf_*\langle v-v_*\rangle^\gamma
    \dd v\dd v_*\\
    &\lesssim E_0\|\hf\|_{L^p_v}^p
    +\int_{\R^d}\hf^p|v|^\gamma\dd v.
\end{aligned}
\end{equation*}
We are left to bound $\|\hf^p|v|^\gamma\|_{L^1_v}$.

By Sobolev embedding in Lemma \ref{lem:sobo}, we have
\begin{align*}
    \|\hf\|_{L^{pk}_v}
    \lesssim
    \|\nabla_v \hf^{\frac p2}\|_{L^2_v}^{\frac2p}
    +\|\hf\|_{L^p_v},
\end{align*}
where $k\in[2,\infty)$ for $d=2$, and
$k=\frac{d}{d-2}$ for $d\ge3$.
When $d\ge3$, the last term can be omitted.

We write
\begin{equation}
\label{rmk}
    \hf^p|v|^\gamma
    =
    \hf^{\frac{pk(p-1)}{pk-1}}
    \hf^{\frac{p(k-1)}{pk-1}}|v|^\gamma.
\end{equation}
By Hölder's inequality with exponents
$\frac{pk-1}{p-1}$ and $\frac{pk-1}{p(k-1)}$, we obtain
\begin{equation*}
\label{gamma+:holder}
\begin{aligned}
    \int_{\R^d}\hf^p|v|^\gamma\dd v
    &\le
    \|\hf\|_{L^{pk}_v}^{\frac{pk(p-1)}{pk-1}}
    \|\hf|v|^{\frac{\gamma(pk-1)}{p(k-1)}}\|_{L^1_v}
    ^{\frac{p(k-1)}{pk-1}}\\
    &\lesssim
    \Big(
    \|\nabla_v\hf^{\frac p2}\|_{L^2_v}
    ^{\frac{2k(p-1)}{pk-1}}
    +\|\hf\|_{L^p_v}^{\frac{pk(p-1)}{pk-1}}
    \Big)
    M_{\frac{\gamma(pk-1)}{p(k-1)}}(f)
    ^{\frac{p(k-1)}{pk-1}}.
\end{aligned}
\end{equation*}
By Young's inequality with exponents
$\frac{pk-1}{k(p-1)}$ and $\frac{pk-1}{k-1}$, we have,
for any $\varepsilon>0$,
\begin{equation*}
\begin{aligned}
    \int_{\R^d}\hf^p|v|^\gamma\dd v
    &\le
    \varepsilon
    \|\langle v\rangle^{\frac{\gamma}{2}}
    \nabla_v\hf^{\frac p2}\|_{L^2_v}^2
    +C\|\hf\|_{L^p_v}^p
    +CM_{\frac{\gamma(pk-1)}{p(k-1)}}(f)^p,
\end{aligned}
\end{equation*}
where $C$ may depend on $\varepsilon$.
Here we use $\gamma>0$.

Notice that, when $2\le d\le4$, we have
$\frac{\gamma(pk-1)}{p(k-1)}\le2$, since
\begin{equation*}
    \frac{\gamma(pk-1)}{p(k-1)}
    =\gamma\Big(1+\frac{p-1}{p(k-1)}\Big)
    \le1+\frac1{k-1}\le2.
\end{equation*}
Thus, the required moment is bounded by the mass and energy bounds.
When $d\ge5$, we have
\begin{equation*}
    \frac{\gamma(pk-1)}{p(k-1)}
    =\gamma\Big(1+\frac{p-1}{p}\frac{d-2}{2}\Big),
\end{equation*}
and the required moment bound follows from Lemma \ref{lem:M-1}
and the assumption on $f_0$.

We choose $\varepsilon$ small enough and combine the estimates
for $A$ and $B$ to derive
\begin{equation*}
\label{tobdd:T}
    \frac{d}{dt}\|\hf\|_{L^p_v}^p
    +C\|\langle v\rangle^{\frac{\gamma}{2}}
    \nabla_v\hf^{\frac p2}\|_{L^2_v}^2
    \lesssim \|\hf\|_{L^p_v}^p+1.
\end{equation*}
Grönwall's inequality implies that
\begin{align*}
    \hf\in L^\infty([0,T];L^p_v),
    \quad
    \langle v\rangle^{\frac{\gamma}{2}}
    \nabla_vF^{\frac p2}
    \in L^2([0,T]\times\R^d).
\end{align*}
Since $f$ is nonnegative, we have
$\|\hf\|_{L^p_v}=\|f\|_{L^p_vL^1_x}$.
Hence,
$f\in L^\infty([0,T];L^p_vL^1_x)$.

\medskip

\item
We follow \cite{ALL15} with appropriate modifications for all
dimensions $d\ge3$ and $p>\frac d2$.
To show the weighted $L^p$-estimates, we define
\begin{equation*}
    h(v)\defeq\langle v\rangle^s\hf(v),
\end{equation*}
which satisfies
\begin{equation}
\label{FL:app:hn}
\begin{aligned}
    \d_t h
    &=\div_v\Big(\int_{\R^d}
    a(\hf_*\nabla_vh-h\nabla_{v_*}\hf_*)\dd v_*\Big)
    -\div_v\Big(
    \Big(\int_{\R^d}a\hf\hf_*\dd v_*\Big)
    \nabla_v\langle v\rangle^s\Big)\\
    &\quad-\nabla_v\langle v\rangle^s\cdot
    \int_{\R^d}a(\hf_*\nabla_v\hf-\hf\nabla_{v_*}\hf_*)\dd v_*.
\end{aligned}
\end{equation}
We test \eqref{FL:app:hn} by $h^{p-1}$ to obtain
\begin{align*}
    p^{-1}\frac{d}{dt}\int_{\R^d}h^p\dd v
    &=
    -(p-1)\int_{\R^{2d}}\hf_*h^{p-2}
    (\nabla_vh)^Ta\nabla_vh\dd v\dd v_*\\
    &\quad+(p-1)\int_{\R^{2d}}h^{p-1}
    (\nabla_vh)^Ta\nabla_{v_*}\hf_*\dd v\dd v_*\\
    &\quad+(p-1)\int_{\R^{2d}}\hf\hf_*h^{p-2}
    (\nabla_vh)^Ta\nabla_v\langle v\rangle^s
    \dd v\dd v_*\\
    &\quad-\int_{\R^{2d}}\hf_*h^{p-1}
    (\nabla_v\langle v\rangle^s)^Ta\nabla_v\hf
    \dd v\dd v_*\\
    &\quad+\int_{\R^{2d}}\hf h^{p-1}
    (\nabla_v\langle v\rangle^s)^Ta\nabla_{v_*}\hf _*
    \dd v\dd v_*\\
    &=:-A+B_1+B_2+B_3+B_4.
\end{align*}
We observe that
\begin{equation*}
    |\nabla_v\langle v\rangle^s|
    \lesssim\langle v\rangle^{s-1},
    \quad
    |\nabla_v^2\langle v\rangle^s|
    \lesssim\langle v\rangle^{s-2}.
\end{equation*}

Similarly to \eqref{A:coe-0}, the coercivity estimate in
Lemma \ref{lem:coer} implies
\begin{align*}
    A\ge
    C\|\nabla_v(\langle v\rangle^{\frac\gamma2}
    h^{\frac p2})\|_{L^2_v}^2
    -C\|h\|_{L^p_v}^p.
\end{align*}

We now bound $|B_i|$, $i=1,\dots,4$.
By integration by parts in $v$ and $v_*$, we have
\begin{align*}
    |B_1|
    &\lesssim
    \Big|\int_{\R^{2d}}h^p\hf _*c(v-v_*)\dd v\dd v_*\Big|\\
    &\lesssim
    \int_{\R^{2d}}h^p\hf _*|v-v_*|^\gamma\dd v\dd v_*.
\end{align*}
For $B_2$, we use $\hf =h\langle v\rangle^{-s}$ and
integrate by parts in $v$ to obtain
\begin{align*}
    |B_2|
    &\lesssim
    \int_{\R^{2d}}h^p\hf _*
    \Big(
    |a|\langle v\rangle^{-2}
    +|b|\langle v\rangle^{-1}
    \Big)\dd v\dd v_*\\
    &\lesssim
    \int_{\R^{2d}}h^p\hf _*|v-v_*|^\gamma\dd v\dd v_*
    +\|h\|_{L^p_v}^p.
\end{align*}
Similarly, integration by parts in $v$ implies
\begin{align*}
    |B_3|
    &\lesssim
    \Big|\int_{\R^{2d}}\hf _*
    (\nabla_v\langle v\rangle^{sp})^Ta\nabla_v\hf ^p
    \dd v\dd v_*\Big|\\
    &\lesssim
    \int_{\R^{2d}}h^p\hf _*
    \Big(
    |a|\langle v\rangle^{-2}
    +|b|\langle v\rangle^{-1}
    \Big)\dd v\dd v_*\\
    &\lesssim
    \int_{\R^{2d}}h^p\hf _*|v-v_*|^\gamma\dd v\dd v_*
    +\|h\|_{L^p_v}^p.
\end{align*}
For $B_4$, integration by parts in $v_*$ implies
\begin{align*}
    |B_4|
    &\lesssim
    \int_{\R^{2d}}h^p\hf _*|b|\langle v\rangle^{-1}
    \dd v\dd v_*\\
    &\lesssim
    \int_{\R^{2d}}h^p\hf _*|v-v_*|^\gamma\dd v\dd v_*
    +\|h\|_{L^p_v}^p.
\end{align*}
Here we use the coefficient bounds, $\gamma<-2$, and
the mass bound.

Combining these estimates, we obtain
\begin{align*}
    \frac{d}{dt}\|h\|_{L^p_v}^p
    +C\|\nabla_v(\langle v\rangle^{\frac\gamma2}
    h^{\frac p2})\|_{L^2_v}^2
    \lesssim
    \|h\|_{L^p_v}^p
    +\int_{\R^{2d}}h^p\hf _*|v-v_*|^\gamma\dd v\dd v_*.
\end{align*}

We split the last integral into
$\{|v-v_*|>1\}$ and $\{|v-v_*|\le1\}$.
On $\{|v-v_*|>1\}$, we have
\begin{align*}
    \int_{\{|v-v_*|>1\}}
    h^p\hf _*|v-v_*|^\gamma\dd v\dd v_*
    \le M_0(f_0 )\|h\|_{L^p_v}^p.
\end{align*}
On $\{|v-v_*|\le1\}$, we use
$\hf ^p\hf _*\le \hf ^{p+1}+\hf _*^{p+1}$ and
$\langle v\rangle^s\lesssim\langle v_*\rangle^s$
to obtain
\begin{align*}
    &\int_{\{|v-v_*|\le1\}}
    h^p\hf _*|v-v_*|^\gamma\dd v\dd v_*\lesssim
    \int_{\R^d}\langle v\rangle^{sp}\hf ^{p+1}\dd v.
\end{align*}

By Sobolev embedding in Lemma \ref{lem:sobo}, we have
\begin{equation}
\label{wl:sobo}
    \|\langle v\rangle^{\frac{\gamma+sp}{2}}
    \hf ^{\frac p2}\|_{L^{2k}_v}^2
    \lesssim
    \|\nabla_v(\langle v\rangle^{\frac{\gamma+sp}{2}}
    \hf ^{\frac p2})\|_{L^2_v}^2,
\end{equation}
where $k=\frac d{d-2}$.
Since $p>\frac d2$, we have $pk-p-k>0$.
Moreover, the assumption $s\ge2+k(|\gamma|-2)$ implies
\begin{align*}
    |\gamma|-\frac{ps}{k(p-1)}
    \le2\frac{pk-p-k}{k(p-1)}.
\end{align*}
It follows that
\begin{align*}
    \langle v\rangle^{sp}\hf ^{p+1}
    \le
    (\langle v\rangle^s\hf )^{\frac{p}{k(p-1)}}
    (\langle v\rangle^{\frac{\gamma+sp}{p}}\hf )^p
    (\langle v\rangle^2\hf )^{\frac{pk-p-k}{k(p-1)}}.
\end{align*}
By Hölder's inequality, using
\begin{align*}
    \frac1{k(p-1)}+\frac1k
    +\frac{pk-p-k}{k(p-1)}=1,
\end{align*}
and \eqref{wl:sobo}, we obtain
\begin{align*}
    \int_{\R^d}\langle v\rangle^{sp}\hf ^{p+1}\dd v
    &\le
    \|h\|_{L^p_v}^{\frac{p}{k(p-1)}}
    \|\langle v\rangle^{\frac{\gamma+sp}{2}}
    \hf ^{\frac p2}\|_{L^{2k}_v}^2
    E_0^{\frac{pk-p-k}{k(p-1)}}\\
    &\lesssim
    \|h\|_{L^p_v}^{\frac{p}{k(p-1)}}
    \|\nabla_v(\langle v\rangle^{\frac\gamma2}
    h^{\frac p2})\|_{L^2_v}^2.
\end{align*}

Let $X=\|h\|_{L^p_v}^p$ and
$Y=\|\nabla_v(\langle v\rangle^{\frac\gamma2}
h^{\frac p2})\|_{L^2_v}^2$.
Combining the above estimates, we obtain
\begin{equation*}
    X'+\big(C_1-C_2X^{\frac1{k(p-1)}}\big)Y
    \le C_2X,
\end{equation*}
for some constants $C_1,C_2>0$.
We choose $\delta>0$ such that
\begin{align*}
    X(0)\le\delta^p
    \le\frac12
    (C_1/2C_2)^{k(p-1)}
    e^{-C_2T}.
\end{align*}
As long as
$X(t)\le(C_1/2C_2)^{k(p-1)}$,
we have
\begin{align*}
    X'+\frac{C_1}{2}Y\le C_2X.
\end{align*}
Indeed, Grönwall's inequality implies
\begin{align*}
    X(t)\le e^{C_2t}X(0)
    \le\frac12
    (C_1/2C_2)^{k(p-1)}
    \quad\forall\,t\in[0,T].
\end{align*}

\end{enumerate}

\end{proof}

\begin{remark}
The method in Lemma \ref{lem:Lp-3} does not directly yield
$L^p$-estimates for the fuzzy Landau equation when
$\kappa\not\equiv1$.
For example, consider the hard potential case $\gamma\in(0,1]$
with $d\ge3$. Let $k=\frac d{d-2}$.
Similarly to \eqref{rmk}, we write
\begin{align*}
    \langle v\rangle^\gamma f^p
    =
    (\langle v\rangle^\gamma f^p)^{\frac{k}{q_1}}
    (\langle v\rangle^rf)^{\frac1{q_2}},
\end{align*}
where
\begin{equation*}
    \frac1{q_1}+\frac1{q_2}=1,
    \quad
    \frac{pk}{q_1}+\frac1{q_2}=p,
    \quad
    \frac{\gamma k}{q_1}+\frac r{q_2}=\gamma.
\end{equation*}
These relations ensure $r=\gamma/p$.
We define the hydrodynamic moments
\begin{align*}
    T_r(f)(x)=\int_{\R^d}\langle v\rangle^rf(x,v)\dd v.
\end{align*}
By Hölder's inequality with exponents $(q_1,q_2)$,
Sobolev embedding, and Young's inequality with exponents
$(q_1/k,q_3)$, we obtain
\begin{align*}
    \int_{\Do}f^p|v|^\gamma\dd x\dd v
    &\lesssim
    \int_x
    \|\nabla_v(\langle v\rangle^{\frac\gamma2}
    f^{\frac p2})\|_{L^2_v}^{\frac{2k}{q_1}}
    T_r(f)^{\frac1{q_2}}\dd x\\
    &\le
    \varepsilon
    \|\nabla_v(\langle v\rangle^{\frac\gamma2}
    f^{\frac p2})\|_{L^2(\Do)}^2
    +\int_xT_r(f)^p\dd x.
\end{align*}
Here $\frac{k}{q_1}+\frac1{q_3}=1$, and the above relations
imply $\frac{q_3}{q_2}=p>1$.
Thus, this argument requires additional control of
\begin{align*}
    \int_0^T\int_xT_r(f_t)^p\dd x\dd t,
\end{align*}
which does not follow from the global moment bounds alone.
Estimating these local moments is more difficult because
the transport term does not vanish.
For $r=0$, at least formally, we have the local mass
conservation law
\begin{align*}
    \d_tT_0(f)
    +\div_x\Big(\int_{\R^d}vf\dd v\Big)=0.
\end{align*}
\end{remark}

We compare the above results with the requirements in \eqref{require}
 \begin{gather*}
        \langle v\rangle^{2+|\gamma|}f
        \in L^\infty([0,T];L^p_vL^1_x),
        \quad \gamma\in(-\gamma_d,0),\\
        f\in L^\infty([0,T];L^p_vL^1_x),
        \quad \gamma\in(0,1].
    \end{gather*}
\begin{remark}
\label{rmk:comparison}
\begin{itemize}
\item \underline{\textit{Soft potentials.}}
For $\gamma\in[-2,0]$, the weighted mixed norms can be
controlled on both $\T^d\times\R^d$ and $\R^{2d}$ under
additional integrability and moment assumptions on the initial
value. 
On the $\T^d_x$, for $p>\frac{d}{d-|\gamma|}$,
Hölder's inequality implies
\begin{align*}
    \|\langle v\rangle^{2+|\gamma|}f\|_{L^p_vL^1_x}
    &\le
    \|\langle v\rangle^{2+|\gamma|}f\|_{L^p(\Dot)}\\
    &\le
    \|f\|_{L^{2p}(\Dot)}^{\frac{2(p-1)}{2p-1}}
    M_{(2+|\gamma|)(2p-1)}(f)^{\frac1{2p-1}}.
\end{align*}
Hence, Lemmas \ref{lem:M-1} and \ref{lem:Lp-2} ensure a uniform
bound on $[0,T]$ if $f_0\in L^{2p}(\Dot)$ and
$M_{(2+|\gamma|)(2p-1)}(f_0)<+\infty$.

On $\R^d_x$, let $a>\frac{d(p-1)}p$ and
$r=(a+2+|\gamma|)(2p-1)$.
Weighted Hölder's inequality and interpolation ensure
\begin{align*}
    \|\langle v\rangle^{2+|\gamma|}f\|_{L^p_vL^1_x}
    &\lesssim_a
    \|\langle x\rangle^a\langle v\rangle^{2+|\gamma|}f
    \|_{L^p(\Do)}\\
    &\lesssim_a
    \|f\|_{L^{2p}(\Do)}^{\frac{2(p-1)}{2p-1}}
    \|f\|_{L^1_{r,r}(\Do)}^{\frac1{2p-1}}.
\end{align*}
Thus, the same bound follows if
$f_0\in L^{2p}(\Do)\cap L^1_{r,r}(\Do)$.
Here the velocity moment propagates by Lemma \ref{lem:M-1},
and the spatial moment is controlled by
\begin{align*}
    \frac{d}{dt}\int_{\Do}\langle x\rangle^rf_t\dd x\dd v
    \lesssim_r
    \int_{\Do}\big(\langle x\rangle^r+\langle v\rangle^r\big)
    f_t\dd x\dd v.
\end{align*}
In particular, these estimates verify the first bound in
\eqref{require} for $\gamma\in[-2,0)$.

For very soft potentials with $\kappa\equiv1$, the local
apriori estimate in Lemma \ref{lem:Lp-3} also implies the
required weighted bound for sufficiently regular solutions,
with $s=2+|\gamma|$ and
$p>\max\{\frac d2,\frac{d}{d-|\gamma|}\}$,
provided the initial weighted norm is sufficiently small
for the prescribed interval $[0,T]$.
A general estimate without this smallness and regularity
restriction remains to be established.

    Related estimates are obtained in
\cite{gualdani2025fuzzylandauequationglobal}.
    In dimension three, for $\gamma\in(-\sqrt8,0]$ and
    $\kappa(x)=\langle x\rangle^{-\lambda}$ for $x\in\R^d$, $\lambda\ge0$,
    the 
    propagation of $L^p_{x,v}$-norms are obtained under stronger initial
    regularity and moment assumptions.
    For $\gamma\le-2$, additional finite initial
    Fisher information is needed.
    
    \smallskip
    \item \underline{\textit{Hard potentials.}}
    When $\kappa\equiv1$, Lemma \ref{lem:Lp-3}
    yields the second bound in \eqref{require} for
    $p\ge2$ and $p>\frac{d}{d-\gamma}$, under the
    initial assumptions stated there.
    This argument does not establish the same bound
    for general bounded spatial kernels.
\end{itemize}
\end{remark}

\begin{remark}[Quantitative mixed-norm estimates for
$\gamma=-2$ and $\kappa\equiv1$ on $\T^d$]
We note that, on $\T^d$, the domains $\Omega_1$ and $\Omega_2$
in the proof of Lemma \ref{lem:Lp-2}-$(2)$ are defined by
$\log f_*<M$ and $\log f_*\ge M$, respectively, without the
spatial weight $|x_*|^2$.
In that proof, we fix $T>0$ and choose $M$ sufficiently large
to ensure the smallness of
$\int_{\{\log f>M\}}\langle v\rangle^2f\dd x\dd v$
in \eqref{omega-2}, uniformly for $t\in[0,T]$.
This choice depends on
$\sup_{t\in[0,T]}M_{2+\varepsilon'}(f_t)$
and therefore may depend on $T$.

In this remark, we follow \cite[Theorem 2]{Wu14} to investigate
a global-in-time quantitative estimate when $\kappa\equiv1$.
We assume $f_0\in L^p_vL^1_x$, $p\in(1,+\infty)$, and
$M_s(f_0)<+\infty$, where $s>6$ if $d=3$, and $s\ge d$
if $d\ge4$.
We first consider small values of $p$ and then use a
bootstrap argument for larger values. We write $\hf(v)=f*_x\kappa=\int_{\T^d}f(x,v)\dd x$.

Let $k=\frac d{d-2}$. In the following, we take 
\begin{align*}
    \varepsilon=\frac{2k(k-1)}{s+2k}
    \in(0,\min(1,k-1)),
    \qquad
    \varepsilon'=\frac{s-2}{2}.
\end{align*}

\begin{itemize}
\item \underline{Step $1$: Small $p\in(1,2-1/(k-\varepsilon))$.}
We have
\begin{equation}
\label{ineq:p}
    1<p<2-\frac1{k-\varepsilon}<k-\varepsilon.
\end{equation}
We recall \eqref{A-B} that
\begin{equation*}
\begin{aligned}
    \frac1p\frac{d}{dt}\int_{\R^d}\hf^p\dd v
    &=
    -(p-1)\int_{\R^d}\hf^{p-2}
    (\nabla_v\hf)^T\overline a\nabla_v\hf\dd v\\
    &\quad-\frac{p-1}{p}
    \int_{\R^d}\hf^p\overline c\dd v
    =:-A+B.
\end{aligned}
\end{equation*}
Repeating the proof of Lemma \ref{lem:Lp-2}-$(2)$, we have
\begin{align*}
    &\int_{\{|v-v_*|\le1\}}
    |v-v_*|^{-2}\hf^p\hf_*\dd v\dd v_*\\
    \lesssim &{}
    e^M\|\hf\|_{L^p_v}^p
    +\|\nabla_v(\langle v\rangle^{-1}\hf^{\frac p2})\|_{L^2_v}^2
    \int_{\{\hf>e^M\}}\langle v\rangle^2\hf\dd v.
\end{align*}

We estimate the last integral. Hölder's inequality implies
\begin{equation}
\label{entropy-tail}
\begin{aligned}
    &\int_{\{\hf>e^M\}}\langle v\rangle^2\hf\dd v\\
    \le&{}
    M^{-\frac{\varepsilon'}{1+\varepsilon'}}
    \int_{\{\hf>e^M\}}
    (\langle v\rangle^{2(1+\varepsilon')}\hf)^{\frac1{1+\varepsilon'}}
    (\hf\log \hf)^{\frac{\varepsilon'}{1+\varepsilon'}}
    \dd v\\
    \le&{}
    M_{2(1+\varepsilon')}(\hf)^{\frac1{1+\varepsilon'}}
    H(\hf)^{\frac{\varepsilon'}{1+\varepsilon'}}
    M^{-\frac{\varepsilon'}{1+\varepsilon'}}.
\end{aligned}
\end{equation}
By Lemma \ref{lem:hydro}, 
$H(\hf)$ is uniformly bounded.
For a sufficiently small constant $\delta>0$, we choose
\begin{align*}
    M(t)=C\Big(
    1+\delta^{-\frac{1+\varepsilon'}{\varepsilon'}}
    M_{2(1+\varepsilon')}(\hf_t)^{\frac1{\varepsilon'}}
    \Big).
\end{align*}
This choice allows us to absorb the high-density contribution
into the coercive term. We obtain
\begin{equation}
\label{-2:qua-1}
    \frac{d}{dt}\|\hf\|_{L^p_v}^p
    +C\|\nabla_v(\langle v\rangle^{-1}\hf^{\frac p2})\|_{L^2_v}^2
    \lesssim e^{M(t)}\|\hf\|_{L^p_v}^p.
\end{equation}

We next estimate $e^{M(t)}\|\hf\|_{L^p_v}^p$.
We write
\begin{align*}
    \hf^p
    =
    (\langle v\rangle^{-1}\hf^{\frac p2})^{\frac{2k}{q_1}}
    \hf^{\frac{k-\varepsilon}{q_2}}
    (\langle v\rangle^r\hf)^{\frac1{q_3}},
\end{align*}
where
\begin{align*}
    \frac{kp}{q_1}+\frac{k-\varepsilon}{q_2}
    +\frac1{q_3}=p,
    \quad
    \frac1{q_1}+\frac1{q_2}+\frac1{q_3}=1,
    \quad
    \frac r{q_3}=\frac{2k}{q_1}.
\end{align*}
For $\alpha\in(0,1)$, we choose
\begin{align*}
    \frac1{q_1}
    =\alpha\frac{p-1}{kp-1},
    \quad
    \frac1{q_2}
    =(1-\alpha)\frac{p-1}{k-\varepsilon-1}.
\end{align*}
The remaining relations determine $q_3$ and $r$.
In particular, we have
\begin{align*}
    0<r<\frac{2k(p-1)}{p(k-1)}<2,
\end{align*}
so $M_r(F)\le E_0$. By Hölder's inequality, we have
\begin{align*}
    e^{M(t)}\|\hf\|_{L^p_v}^p
    \le
    \|\langle v\rangle^{-1}\hf^{\frac p2}\|_{L^{2k}_v}^{\frac{2k}{q_1}}
    \|\hf\|_{L^{k-\varepsilon}_v}^{\frac{k-\varepsilon}{q_2}}
    M_r(\hf)^{\frac1{q_3}}e^{M(t)}.
\end{align*}
We choose Young exponents satisfying
\begin{align*}
    \frac1a=\frac{k}{q_1},
    \quad
    \frac1b=\frac{k-\varepsilon}{q_2},
    \quad
    \frac1a+\frac1b+\frac2c=1.
\end{align*}
By \eqref{ineq:p}, such exponents can be chosen strictly
larger than one. Young's inequality and Sobolev embedding ensure,
for sufficiently small $\delta>0$,
\begin{align*}
    e^{M(t)}\|\hf\|_{L^p_v}^p
    &\le
    \delta\|\langle v\rangle^{-1}\hf^{\frac p2}\|_{L^{2k}_v}^2
    +C\|\hf\|_{L^{k-\varepsilon}_v}
    +CM_r(\hf)^{\frac c{q_3}}+Ce^{cM(t)}\\
    &\lesssim
    \delta\|\nabla_v(\langle v\rangle^{-1}\hf^{\frac p2})\|_{L^2_v}^2
    +\|\hf\|_{L^{k-\varepsilon}_v}+1+e^{cM(t)}.
\end{align*}
Combining this estimate with \eqref{-2:qua-1}, we obtain
\begin{align*}
    &\frac{d}{dt}\|\hf\|_{L^p_v}^p
    +C\|\nabla_v(\langle v\rangle^{-1}\hf^{\frac p2})\|_{L^2_v}^2
    \lesssim
    \|\hf\|_{L^{k-\varepsilon}_v}+1+e^{cM(t)}.
\end{align*}

The choices of $\varepsilon$ and $\varepsilon'$ ensure
\begin{align*}
    s=2k\Big(\frac{k-1}{\varepsilon}-1\Big)
    =2(1+\varepsilon'),
\end{align*}
which allow us to take the same moment in \eqref{entropy-tail} and in
Lemma \ref{lem:sobo}-$(5)$.
The moment bound implies
\begin{align*}
    M(t)
    \lesssim 1+M_s(f_t)^{\frac{2}{s-2}}
    \lesssim
    \langle t\rangle^{\max\{\frac{2}{s-2},\frac23\}}.
\end{align*}
Moreover, Lemmas \ref{lem:sobo} and \ref{lem:M-1},
together with the entropy inequality, imply at most polynomial
growth of $\int_0^t\|\hf_\tau\|_{L^{k-\varepsilon}_v}\dd\tau$.
Indeed, by Minkowski's inequality and
Lemma \ref{lem:sobo}-$(5)$, we have
\begin{align*}
    \|\hf\|_{L^{k-\varepsilon}_v}
    &\le \|f\|_{L^1_xL^{k-\varepsilon}_v}\\
    &\le
    \|\langle v\rangle^\gamma f\|_{L^1_xL^k_v}^{\frac{k(k-1-\varepsilon)}{(k-\varepsilon)(k-1)}}
    M_s(f)^{\frac{\varepsilon}{(k-\varepsilon)(k-1)}}.
\end{align*}
Integrating the preceding differential inequality, we obtain
\begin{align*}
    \|\hf_t\|_{L^p_v}^p
    \lesssim
    \exp\Big(C\langle t\rangle^{
    \max\{\frac{2}{s-2},\frac23\}}\Big).
\end{align*}

\smallskip

\item \underline{Step $2$: Bootstrap to larger values of $p$.} Starting from the bound established in Step $1$, we bootstrap
from $p$ to $p+\delta$, where
$0<\delta<\frac{k-1}{kp}$ and
$\hf_0\in L^{p+\delta}(\R^d)$.
Repeating the arguments in step $1$,
we obtain
\begin{equation}
\label{-2:qua-2}
\begin{aligned}
    &\frac{d}{dt}\|\hf\|_{L^{p+\delta}_v}^{p+\delta}
    +C_0\|\nabla_v(\langle v\rangle^{-1}
    \hf^{\frac{p+\delta}{2}})\|_{L^2_v}^2\lesssim
    e^{M(t)}\|\hf\|_{L^{p+\delta}_v}^{p+\delta}.
\end{aligned}
\end{equation}

We write
\begin{align*}
    \hf^{p+\delta}
    =
    (\langle v\rangle^{-1}\hf^{\frac{p+\delta}{2}})^{\frac{2k}{q_1}}
    \hf^{\frac p{q_2}}
    (\langle v\rangle^r\hf)^{\frac1{q_3}},
\end{align*}
where
\begin{align*}
    \frac{k(p+\delta)}{q_1}+\frac p{q_2}
    +\frac1{q_3}=p+\delta,
    \quad
    \frac1{q_1}+\frac1{q_2}+\frac1{q_3}=1,
    \quad
    \frac r{q_3}=\frac{2k}{q_1}.
\end{align*}
For $\alpha=\frac{p-1}{p}$, we choose
\begin{align*}
    \frac1{q_1}
    &=
    \frac{(1-\alpha)(p-1)+\delta}{k(p+\delta)-1},
    \quad
    \frac1{q_2}=\alpha,\\
    \frac1{q_3}
    &=
    \frac{k-1+(\frac kp-1)\delta}{k(p+\delta)-1}.
\end{align*}
We note that
\begin{align*}
    1-\frac{k}{q_1}-\frac1{q_2}>0,
    \quad
    r=
    \frac{2k(p-1+p\delta)}
    {p(k-1)+(k-p)\delta}
    <\frac{2k}{k-1}=d.
\end{align*}

By Hölder's inequality, we have
\begin{align*}
    &e^{M(t)}\|\hf\|_{L^{p+\delta}_v}^{p+\delta}\\
    \le&{}
    \|\langle v\rangle^{-1}
    \hf^{\frac{p+\delta}{2}}\|_{L^{2k}_v}^{\frac{2k}{q_1}}
    \|\hf\|_{L^p_v}^{\frac p{q_2}}
    M_r(\hf)^{\frac1{q_3}}e^{M(t)}.
\end{align*}
We take $a=q_1/k$, $b=q_2$, and choose $c>1$ such that
\begin{align*}
    \frac2c=1-\frac{k}{q_1}-\frac1{q_2}.
\end{align*}
Young's inequality and Sobolev embedding then imply
\begin{align*}
    &e^{M(t)}\|\hf\|_{L^{p+\delta}_v}^{p+\delta}\\
    \le&{}
    \frac{C_0}{2}
    \|\nabla_v(\langle v\rangle^{-1}
    \hf^{\frac{p+\delta}{2}})\|_{L^2_v}^2
    +C\|\hf\|_{L^p_v}^p
    +CM_r(\hf)^{\frac c{q_3}}+Ce^{cM(t)}.
\end{align*}
Combining this estimate with \eqref{-2:qua-2}, we obtain
\begin{align*}
    &\frac{d}{dt}\|\hf\|_{L^{p+\delta}_v}^{p+\delta}
    +C\|\nabla_v(\langle v\rangle^{-1}
    \hf^{\frac{p+\delta}{2}})\|_{L^2_v}^2\\
    \lesssim&{}
    \|\hf\|_{L^p_v}^p+M_r(\hf)^{\frac c{q_3}}+e^{cM(t)}.
\end{align*}
Since $r<d\le s$, the assumption $M_s(f_0)<+\infty$
and Lemma \ref{lem:M-1} ensure at most polynomial growth of
$M_r(\hf_t)$.
Using the established $L^p_v$-bound and the estimate for $M(t)$,
we integrate in time to obtain
\begin{align*}
    \|\hf_t\|_{L^{p+\delta}_v}^{p+\delta}
    \lesssim
    \exp\Big(C\langle t\rangle^{
    \max\{\frac{2}{s-2},\frac23\}}\Big).
\end{align*}
Repeating this argument implies the estimate for every finite
$p>1$ for which $F_0\in L^p(\R^d)$.
Since $\|\hf_t\|_{L^p_v}=\|f_t\|_{L^p_vL^1_x}$,
we obtain the corresponding mixed-norm estimates.

\end{itemize}
\end{remark}

\printbibliography

\Addresses

\end{document}